\documentclass[12pt]{article}

\usepackage[utf8]{inputenc}
\usepackage{epsfig}
\usepackage{amsmath}
\usepackage{amsfonts}
\usepackage{amsthm}
\usepackage{graphicx,subfigure}
\usepackage{color}

        \theoremstyle{plain}

        \newtheorem{prop}{Property}[section]

        \theoremstyle{remark}
        \newtheorem{remark}{\bf Remark}[section]
        \theoremstyle{remark}
        
        \theoremstyle{remark}

\newcommand{\demi}{\frac{1}{2}}

\newcommand{\tdemi}{\textstyle \frac{1}{2}}

\newcommand{\R}{{\mathbb{R}}}

\newcommand{\dt}{\partial_t}
\newcommand{\dx}{\partial_x}

\newcommand{\cint}[1]{\left\langle #1\right\rangle}
\newcommand{\Dv}{\Delta v}

\newcommand{\Dx}{\Delta x}
\newcommand{\Dt}{\Delta t}

\newcommand{\ximdemi}{x_{i-\frac{1}{2}}}
\newcommand{\xipdemi}{x_{i+\frac{1}{2}}}
\newcommand{\taunpuni}{\tau_i^{n+1}}

\newcommand{\fni}{f^n_i}
\newcommand{\fnpuni}{f^{n+1}_i}
\newcommand{\fnipun}{f^n_{i+1}}
\newcommand{\fnimun}{f^n_{i-1}}
\newcommand{\fnik}{f_{i,k}^n}

\newcommand{\fnpunik}{f_{i,k}^{n+1}}

\newcommand{\Vni}{{\cal V}^n_i}
\newcommand{\vminni}{v^n_{min,i}}
\newcommand{\vmaxni}{v^n_{max,i}}
\newcommand{\vnik}{v^n_{i,k}}
\newcommand{\Vnipun}{{\cal V}^n_{i+1}}
\newcommand{\Vnpuni}{{\cal V}^{n+1}_i}
\newcommand{\vminnpuni}{v^{n+1}_{min,i}}
\newcommand{\vmaxnpuni}{v^{n+1}_{max,i}}
\newcommand{\vnpunik}{v^{n+1}_{i,k}}
\newcommand{\Dvni}{\Delta v^n_i}
\newcommand{\Dvnpuni}{\Delta v^{n+1}_i}
\newcommand{\Uni}{U^n_i}
\newcommand{\Unpuni}{U^{n+1}_i}
\newcommand{\Unpunibar}{U^{n+1,\star}_i}
\newcommand{\Unibar}{U^{n,\star}_i}

\newcommand{\unpuni}{u^{n+1}_i}
\newcommand{\Tnpuni}{T^{n+1}_i}
\newcommand{\cintni}[1]{\left\langle #1\right\rangle_{\Vni}}
\newcommand{\cintnipun}[1]{\left\langle #1\right\rangle_{\Vnipun}}
\newcommand{\cintnpuni}[1]{\left\langle #1\right\rangle_{\Vnpuni}}
\newcommand{\K}{K}
\newcommand{\fnibar}{\bar{f}^n_i}
\newcommand{\fnipunbar}{\bar{f}^n_{i+1}}
\newcommand{\fnimunbar}{\bar{f}^n_{i-1}}
\newcommand{\Kn}{\text{Kn}}

\newcommand{\modif}[1]{#1}

\begin{document}

\begin{center}

{\bf Local discrete velocity grids for deterministic rarefied flow simulations }

\vspace{1cm}
S. Brull$^1$, L. Mieussens$^2$

\bigskip
$^1$Univ. Bordeaux, IMB, UMR 5251, F-33400 Talence, France.\\
CNRS, IMB, UMR 5251, F-33400 Talence, France. \\
({\tt Stephane.Brull@math.u-bordeaux1.fr})

\bigskip
$^2$Univ. Bordeaux, IMB, UMR 5251, F-33400 Talence, France.\\
CNRS, IMB, UMR 5251, F-33400 Talence, France. \\
INRIA, F-33400 Talence, France.\\
({\tt Luc.Mieussens@math.u-bordeaux1.fr})

\end{center}

\begin{abstract}
Most of numerical methods for deterministic simulations of rarefied
gas flows use the discrete velocity (or discrete ordinate)
approximation. In this approach, the kinetic equation is approximated
with a global velocity grid.  The grid must be large and fine enough
to capture all the distribution functions, which is very expensive for
high speed flows (like in hypersonic aerodynamics). In this article,
we propose to use instead different velocity grids that are local in
time and space: these grids dynamically adapt to the width of the
distribution functions. The advantages and drawbacks of the method are
illustrated in several 1D test cases.
\end{abstract}
\\

\section{Introduction}

Most of deterministic numerical methods for rarefied flow simulations
are based on a discrete velocity approximation of the Boltzmann
equation, see for instance~\cite{TSA1993,rogier-schneider, buet_ttsp,
  BPS, BR_1999,luc_jcp,PH2002,KAAFZ,T_cicp,HG2012}.

In almost all these methods, the distribution function is approximated
with a global velocity grid, for every point in the position space,
for every time. This makes the method robust (conservation, entropy
dissipation, positivity, stability, etc.) and relatively simple, but
very expensive for many cases. Indeed, the grid must be large enough
to contain all the distribution functions of the flow, and fine enough
to capture every narrow distributions. The first constraint makes the
grid very large for high speed flow with large temperatures. The
second constraint makes the grid step very small, and hence a very
large number of discrete velocities are needed. This
is for instance the case for atmospheric re-entry problems, where the
flow is hypersonic. These problems, especially in 3D, are very
difficult to be simulated with such methods, due to the discrete
velocity grid that contains a prohibitively large number of points.

Of course, particle solvers like the popular Direct Simulation
  Monte-Carlo method (DSMC) do not suffer of such
problems~\cite{bird}. However, if one is interested in deterministic
Eulerian simulations, it is important to find a way to avoid the use
of a too large number of discrete velocities. Up to our knowledge,
there are a few papers on this subject. Aristov proposed
in~\cite{Aristov77} an adaptive velocity grid for the 1D shock
structure calculation. However, the approach is very specific to this
test case and has never been extended. More recently, Filbet and
Rey~\cite{FR2013} proposed to use a rescaling of the velocity variable
to make the support of the distribution independent of the temperature
and of the macroscopic velocity. Then the Boltzmann equation is
transformed into a different form (with inertia terms due to the
change of referential). In~\cite{BCHM_saragosse}, Baranger et
al. proposed an algorithm to locally refine the velocity grid wherever
it is necessary and to coarsen it elsewhere. But this approach, which
has been proved to be very efficient for steady flows, is still based
on a global grid, and cannot be efficient enough for unsteady flows.
Finally, Chen et al.~\cite{CXLC2012} proposed to use a different
velocity grid for each point in the position space and every time:
from one point to another one, the grid is refined or coarsen by using
an Adaptive Mesh Refinement (AMR) technique. This seems to be
very efficient, but all the grids have the same bounds (they all use
the same background grid). It seems that a similar method was proposed
at the same time by Kolobov et al., see~\cite{KAF2011,KA2012}.

In this paper, we propose a method that has several common features
with the method of~\cite{XH2010} and~\cite{FR2013} but is still very
different. The main difference is that each distribution is
discretized on its own velocity grid: each grid has its own bounds and
step that are evolved in time 
and space by using the local macroscopic velocity and
temperature. These macroscopic quantities are estimated by solving the
local conservation laws of mass, momentum, and energy. 
The
interaction between two space cells requires to use reconstruction
techniques to approximate a distribution on different velocity
grids. This paper is a preliminary work, for 1D flows, that proposes a
complete algorithmic approach. Several test cases illustrate the
properties of the method and show its efficiency.

The outline of the paper is the following. In section~\ref{sec:model} is
presented a simple 1D kinetic Bhatnagar-Gross-Krook (BGK) model and its standard discrete
velocity approximation. In section~\ref{sec:LDVmethod}, the local discrete
velocity grid approach is described. The numerical tests are given in section~\ref{sec:res}.

\section{A simple 1D kinetic model and its standard velocity discretization}
\label{sec:model}

\indent We consider a one-dimensional gas described by the mass
density of particles $f(t,x,v)$ that at time $t$ have the position $x$
and the velocity $v$ (note that both position $x$ and velocity
  $v$ are scalar). The corresponding macroscopic quantities can be
obtained by the moment vector $U(t,x)=\cint{mf(t,x,.)}$, where
$m(v)=(1,v,\tdemi |v|^2)$ and $\cint{\phi}=\int_\R \phi(v)\, dv$ for
any velocity dependent function. This vector can be written component
wise by $U=(\rho,\rho u, E)$, where $\rho$, $\rho u$, and $E$ are the
mass, momentum, and energy densities. The temperature $T$ of the gas
is defined by relation $E=\tdemi \rho |u|^2+\tdemi \rho R T$, where
$R$ is the gas constant.

The evolution of the gas is governed by the following BGK equation
\begin{equation}  \label{eq-bgk}
\dt f + v\dx f = \frac{1}{\tau}(M(U)-f), 
\end{equation}
where $M(U)$ is the local Maxwellian distribution defined through the
macroscopic quantities $U$ of $f$ by
\begin{equation}   \label{eq-maxw}
M(U)=\frac{\rho}{\sqrt{2\pi R T}}\exp\bigl(-\frac{|v-u|^2}{2RT}\bigr),
\end{equation}
and $\tau=CT^\omega/\rho$ is the relaxation time. The constant $C$ and
$\omega$ will be given in section \ref{sec:res} for each test case.

From this equation, it
is easy to establish the so called conservation laws that describe the
time evolution of the moment vector $U$:
\begin{equation}  \label{eq-cons}
\dt U + \dx \cint{v mf} = 0.
\end{equation}

For the numerical approximation of equation~(\ref{eq-bgk}), a popular
method is the discrete ordinate--or discrete velocity--method. It
consists in choosing a grid ${\cal V}$ of $K$ points $v_k$, and then in
replacing the kinetic equation~(\ref{eq-bgk}) by the finite set of $K$
equations
\begin{equation}  \label{eq-bgkk}
\dt f_k + v_k\dx f_k = \frac{1}{\tau}(M_k(U)-f_k)
\end{equation}
where $f_k(t,x)$ is an approximation of $f(t,x,v_k)$ and $M_k(U)$ is
an approximation of $M(U)(v_k)$.  This approximation
  is the discrete Maxwellian whose parameters are such that it has the
  same discrete moments as the distribution $f$, as proposed
  in~\cite{luc_m3as,luc_jcp}. This gives a conservative discrete
  velocity model.

In order to describe the solution correctly, the discrete velocity grid
${\cal V}$ must capture all the distribution functions, that is to say
at any time $t$, and for every position $x$. This means that ${\cal
  V}$ must be large enough to capture distributions with large mean
velocity or large temperature, and fine enough to capture
distributions with small temperature. See an illustration of this
problem in figure~\ref{fig:hermes}. In this figure, we show a 2D
aerodynamical flow with three typical distributions functions (one in
the upstream flow, one in the shock, and another one at the
boundary). The corresponding velocity grid is shown in the same
figure.

To construct such a grid,
we start with a remark on the local
  Maxwellians. Since a Maxwellian centered on $u$ and of temperature
  $T$ decreases very fast for large $v$, it is very small outside any
  interval $[u-c\sqrt{RT},u+c\sqrt{RT}]$ with $c$ sufficiently
  large. A good choice for such an interval is obtained with
  $c=3\sqrt{RT}$: as it is well know in statistics for the normal
  distribution, 99\% of the particles described by the local
  Maxwellian have their velocity in this interval. For kinetic
  simulations, we generally take a slightly larger interval with
  $c=4$. The corresponding interval $[u-4\sqrt{RT},u+4\sqrt{RT}]$ is
  what we call the ``support'' of the local Maxwellian and contains
  the ``essential'' information on the distribution.
When a distribution is not too far from its corresponding local
Maxwellian (which is true when the Knudsen number is not too small,
away from shock and boundary layers), most of particles described by
this distribution have their velocity localized in the support of the corresponding
local Maxwellian. This is interesting, since this support can be
analytically determined as a function of the macroscopic velocity and
temperature, as it has been shown above. 

Consequently, a first constraint for the global velocity grid is
that its bounds $v_{min}$ and $v_{max}$ satisfy the following
inequalities:
\begin{equation}\label{eq-gbound} 
  v_{min} \leq \min_{t,x}\bigl(u(t,x)-4\sqrt{RT(t,x)}\bigr), \qquad 
 v_{max} \geq \max_{t,x}\bigl(u(t,x)+4\sqrt{RT(t,x)}\bigr),
\end{equation}
so that all the distributions can be captured in the grid.
Since it is reasonable to require that there are at least
three points between the inflexion points of any Maxwellian, the step
of the global grid should be such that
\begin{equation}\label{eq-gstep} 
  \Dv \leq \min_{t,x}\sqrt{RT(t,x)}.
\end{equation}
Of course, such an approach requires to first estimate some bounds on
the macroscopic fields that are global in time and space.

Note that the points of ${\cal V}$ are not necessarily uniformly
distributed, since the grid could be refined only wherever it is
necessary and coarsened elsewhere, as it is proposed
in~\cite{BCHM_2012} for steady flows (with a simple and automatic way
to define such a grid). However, for unsteady problems, the situation
is more complex. Indeed, first, the estimations of the correct bounds
and step of the grid are not necessarily available for every problems
(the velocity or the temperature could reach much larger values that
were not expected at some times of the simulation), like in complex
shock interactions, for instance. Moreover, even if it is possible,
this could lead to a grid which is extremely large and dense, hence
leading to a very expensive simulation (see an example in
section~\ref{subsec:blast}).
Finally, when there are distributions that are very far from their local Maxwellian, their
support can be quite different, and there is no analytical way to
determine it. This can require several tries  to
find a correct global velocity grid, which is also expensive.

It is therefore very attractive to try to use a {\it local} velocity
discretization of the distribution function, which means to use {\it
  local discrete velocity grids} (LDV) for each time and position. In
other words,\label{review-page8} at each time $t$ and position $x$, we would like
  the corresponding distribution $f(t,x,.)$ to be approximated on its
  own velocity grid, which might be different from the grids used at
  other times or other positions. The clear advantage of this idea is
that we can define an optimally small grid for each distribution, thus
we avoid the problems mentioned above. This approach is developed in
the next section.

\section{A local discrete velocity grid approach}
\label{sec:LDVmethod}


   \subsection{The method}
   \label{subsec:ldv}
\indent   Now we assume that at time $t_n$, the distribution function in each
   space cell $[\ximdemi,\xipdemi]$ is approximated on a set $\Vni$ of
   $\K$ local discrete velocities. For simplicity, we assume here that all the local
   grids have the same number of points $K$, and the points are
   uniformly distributed. The first point is denoted by $\vminni$ and the
   last one by $\vmaxni$. That is to say, we have
\begin{equation*}
  \Vni=\left\lbrace\vnik=\vminni+(k-1)\Dvni, \quad k 
\text{ from } 1 \text{ to } \K 
\right\rbrace,
  \quad \text{ where } \quad \Dvni=\frac{\vmaxni-\vminni}{\K-1}.
\end{equation*}
On this local grid, $f(t_n,x_i,.)$ is approximated by $\K$ values that
are stored in the vector $\fni=(\fnik)_{k=1}^\K$. Each discrete value
$\fnik$ is an approximation of $f(t_n,x_i,\vnik)$. 

The corresponding macroscopic quantities $U(t_n,x_i)$
are approximated by $\Uni$ with the following
quadrature formula
\begin{equation}\label{eq-Uni} 
  \Uni=\cintni{m\fni}=\sum_{k=1}^\K m(\vnik) \fnik \omega_k,
\end{equation}
where the $\omega_k$ are the weights of the quadrature. 
In this paper, the trapezoidal quadrature formula is
  used: $\omega_1=\omega_K=1/2$ and $\omega_k=1$ for $k$ from 2 to $K-1$.

When one wants to compute an approximation of $f$ at the next time
step $t_{n+1}$, two problems occur. First, how to determine the local
discrete velocity grid $\Vnpuni$? We will show below that this can be
simply made by using the conservation laws. Second, how to exchange
information between two space cells, since the local grids are not the
same? This is where we use some interpolation procedure in the
method. Let us now describe our algorithm step by step. 

\underline{\it Step 1:} Macroscopic quantities at $t_{n+1}$. 

\smallskip

We approximate the conservation relation~(\ref{eq-cons}) with the
following first order upwind scheme
\begin{equation}  \label{eq-Unpunibar}
\Unpuni=\Uni -\frac{\Dt}{\Dx}\left(\Phi^n_{i+\demi}- \Phi^n_{i-\demi}\right),
\end{equation}
where the numerical fluxes are defined by
\begin{equation}  \label{eq-numflux}
\Phi^n_{i+\demi}=\cintni{v^+m\fni} + \cintnipun{v^-m\fnipun} ,
\end{equation}
which is indeed an approximation of the flux
$\cint{vmf(t_n,x_{i+\demi})}$ at the cell interface. Here, we use the
standard notation $v^\pm=(v\pm|v|)/2$. Note that each
half flux is computed on the local velocity grid of the corresponding
distribution. The vector $\Unpuni$ is an approximation of
$U(t_{n+1},x_i)$, and we note $\unpuni$ and $\Tnpuni$ the
corresponding velocity and temperature.\\

\underline{\it Step 2:} discrete velocity grid $\Vnpuni$. 

\smallskip

We define this grid by using the new velocity and temperature
$\unpuni$ and $\Tnpuni$ to get the bounds
\begin{equation}  \label{eq-boundsnpuni}
\vminnpuni=\unpuni-4\sqrt{R\Tnpuni} \quad \text{ and } \quad 
\vmaxnpuni=\unpuni+4\sqrt{R\Tnpuni}.
\end{equation}
Then the new grid $\Vnpuni$ is defined as in the previous time step,
that is to say by  
\begin{equation}\label{eq-Vnpuni} 
\begin{split}
 \Vnpuni=\lbrace\vnpunik & =  \vminnpuni+(k-1)\Dvnpuni, 
\quad
 k \text{ from } 1 \text{ to } \K  \rbrace,\\
& \text{ with }\Dvnpuni=(\vmaxnpuni-\vminnpuni)/(\K-1). 
\end{split}
\end{equation}

\underline{\it Step 3:} distribution function at time $t_{n+1}$. 

\smallskip

Here, equation~(\ref{eq-bgk}) is approximated by a first order
upwind scheme, with an implicit relaxation term. If the velocity
variable is not discretized, we get for every $v$:
\begin{equation*}
\begin{split}
  \fnpuni(v)=\fni(v) & - \frac{\Dt}{\Dx}v^+(\fni(v)-\fnimun(v)) 
- \frac{\Dt}{\Dx}v^-(\fnipun(v)-\fni(v))  \\
& + \frac{\Dt}{\taunpuni}(M(\Unpuni)(v)-\fnpuni(v)).
\end{split}
\end{equation*}
If now we take into account that each distribution $\fnpuni$, $\fni$,
$\fnimun$, and $\fnipun$ are defined on their own local velocity grid,
this scheme must be modified by using a reconstruction procedure. 

The discrete distributions $\fni$,
$\fnimun$, and $\fnipun$ are used to reconstruct piecewise continuous
in velocity functions $\fnibar$,
$\fnimunbar$, and $\fnipunbar$ that are
defined as follows:
\begin{equation}  \label{eq-deffnibar}
\fnibar(v)=\left\lbrace 
\begin{split}
& p^n_i(v) \quad \text{ if } \vminni \leq v \leq \vmaxni\\
& 0 \quad \text{ else},
\end{split}
\right.
\end{equation}
where $p^n_i$ is a piecewise continuous function of $v$ constructed
through the values $(\vnik,\fnik)_{k=1}^\K$, like a piecewise
interpolated polynomial. {Since any kind of interpolation could be
  used, }this reconstruction will be discussed in section
\ref{interpol}. Then we define the discrete values of $\fnpuni$ on its
grid $\Vnpuni$ by
\begin{equation}  \label{eq-fnpunik}
\begin{split}
  \fnpunik=\fnibar(\vnpunik) & - \frac{\Dt}{\Dx}{\vnpunik}^+(\fnibar(\vnpunik)-\fnimunbar(\vnpunik)) 
- \frac{\Dt}{\Dx}{\vnpunik}^-(\fnipunbar(\vnpunik)-\fnibar(\vnpunik))  \\
& + \frac{\Dt}{\taunpuni}(M_k(\Unpuni)-\fnpunik),
\end{split}
\end{equation}
for $k=1$ to $\K$. 

Our scheme is then given by
relations~(\ref{eq-Unpunibar}--\ref{eq-fnpunik}). Now, we give some
properties of this scheme.

\begin{prop} \label{prop:cons} 
  \modif{For scheme~(\ref{eq-Unpunibar}--\ref{eq-fnpunik})}, the global mass, momentum, and energy are
  constant (the scheme is conservative): 
\begin{equation*}
\sum_i\Uni \Dx  = \sum_i U^0_i\Dx.
\end{equation*}
\end{prop}
\begin{proof}
  This is a direct consequence of the discretization of the
  conservation laws~(\ref{eq-Unpunibar}) with a conservative scheme:
  indeed, when we take the sum of~(\ref{eq-Unpunibar}) for every $i$,
  the sum of the numerical fluxes cancels out, and we obtain that the
  total quantities do not change during one time step. This gives the result.
\end{proof}
\modif{Even if this property is
obvious, we believe it deserves to be noted: first, we point out that
the scheme is not given by~(\ref{eq-Unpunibar}) only, but by all the relations between
~(\ref{eq-Unpunibar}) and~(\ref{eq-fnpunik}). Then relation~(\ref{eq-Unpunibar}) has to be seen as macroscopic
conservation laws in which the fluxes are computed by using the
discrete kinetic equation~(\ref{eq-fnpunik}). Even if~(\ref{eq-fnpunik}) is not conservative, the
use of~(\ref{eq-Unpunibar}) implies that the macroscopic mass, momentum, and energy
densities are conserved. This is a property shared by several recent
schemes based on a dual macro-micro time evolutions or IMEX methods,
see for instance~\cite{PP2007,BLM_2008,DP_2013}.
}

\begin{prop} \label{prop:pos}
Assume that for every cell $i$, $\fni$ is non negative at each point of
  its local velocity grid and  that the corresponding reconstructed
  piecewise function $\fnibar$ is non negative for every
  $v$. Then, under the CFL condition $\Dt\leq
  \Dx/\max_{i,k}(|\vnpunik|)$, $\fnpunik$ is non negative at each point
  of its local velocity grid, for every space cell.
\end{prop}

\begin{proof}
  As it is standard for the upwind scheme for convection problems, it
  is sufficient to note that~(\ref{eq-fnpunik}) can be written as a
  linear combination of $\fnibar$,
$\fnimunbar$, $\fnipunbar$, and $M(\Unpuni)$. The CFL condition of
the proposition ensures that the coefficients of this combination are
non negative, which gives the result.
\end{proof}

We point out that, while the result of this property is rather
  standard, it is in fact quite weak here. Indeed, first, the non
  negativeness of the distribution is obtained only if the
  reconstruction step preserves the non negativeness of the
  $\fni$. This is true for linear interpolation, but it is not for
  many higher order reconstructions. 
Moreover, this property itself is not sufficient to ensure that
  the sequence $\fni$ can be defined at every time step: indeed, step
  2 requires $\Tnpuni$ to be non-negative to define the local grid
  $\Vnpuni$. Unfortunately, it seems hardly possible to prove that this property is
true for the scheme presented above. This why a modified schemes are
presented in the next section.

Finally, we want to comment on the choice of the time step in this
scheme. Indeed, note that, according to property~\ref{prop:pos},
step 3 requires a time step defined through the local grids $\Vnpuni$
at time $t_{n+1}$ to correctly define $\fnpuni$. However, this time
step is already needed at step 1 to define $\Unpuni$, while at this
step, $\Vnpuni$ is not already known. This means that we have to use a
single time step for steps 1 and 3 that also satisfies the CFL
condition based on $\Vnpuni$. A simple algorithm is the following:
\begin{enumerate}
\item[(a)] We choose $\Dt_1= \Dx/\max_{i,k}(|\vnik|)$
\item[(b)] We do step 1 and  step 2.
\item[(c)] For step 3, we compute $\Dt_2= \Dx/\max_{i,k}(|\vnpunik|)$: 
  \begin{itemize}
  \item if   $\Dt_2>\Dt_1$, then $f^{n+1}$ can be advanced with $\Dt_1$
  \item if $\Dt_2<\Dt_1$, then we set $\Dt_1=\Dt_2$, we do not compute
    $f^{n+1}$ but we directly go back to (b) (steps 1 and 2 of the
    scheme are
  done again)
  \end{itemize}
\end{enumerate}

However, note that in practice, we do not need this algorithm: we
always use $\Dt_1= \Dx/\max_{i,k}(|\vnik|)$ without any stability problem.
\modif{Indeed, we carefuly checked the sign of the solution at each time step
and at every space and velocity point, for all the test cases
presented in this paper: we did not observe any loss of positivity and
any stability problem.}

   \subsection{Modified versions of the scheme}
   \label{subsec:modif}

First, note that if we compute the moments of $\fnpuni$ after step 3,
we do not recover the moments $\Unpuni$ defined at step 1. Indeed,
according to~(\ref{eq-fnpunik}) we have
\begin{equation*}
\begin{split}
& \cintnpuni{m \fnpuni}  = \sum_{\Vnpuni}m(\vnpunik)\fnpunik\omega_k \\
& = \sum_{\Vnpuni}m(\vnpunik)\fnibar(\vnpunik) \omega_k 
    - \frac{\Dt}{\Dx} \left(\sum_{\Vnpuni}m(\vnpunik)({\vnpunik}^+\fnibar(\vnpunik)+ {\vnpunik}^-\fnipunbar(\vnpunik)) \omega_k \right.\\
 & \hspace{34ex}    \left.      - \sum_{\Vnpuni}m(\vnpunik) ({\vnpunik}^+\fnimunbar(\vnpunik)  - {\vnpunik}^- \fnibar(\vnpunik))\omega_k \right)\\
& \modif{\qquad + \frac{\Dt}{\taunpuni}\left(\Unpuni-\cintnpuni{m\fnpunik}\right)}. 
\end{split}
\end{equation*}
If we compare the terms of this expression to the definition of $\Unpuni$ given
by~(\ref{eq-Unpunibar}) and~(\ref{eq-numflux}), we find the two
vectors cannot be equal. The reason is that in the first expression,
we have several quantities on the form
$\sum_{\Vnpuni}\phi(\vnpunik)\fnibar(\vnpunik)\omega_k$, while in the
second expression, these quantities are
$\sum_{\Vni}\phi(\vnik)\fnik\omega_k$, and they are not equal in
general since the grids $\Vnpuni$ and $\Vni$ are different. Of course, these
quantities are close, since they approximate the same values, but
they are not equal. 

This means that we have two different approximations of the same
macroscopic values: $\Unpuni$ and the moments of $\fnpuni$. We have
numerically compared these quantities and there is indeed no
significant difference. However, this difference suggests a
modification of the scheme: after step 3, we add one more step in
which we define
$\Unpunibar=\cintnpuni{m \fnpuni} $, and $\Uni$ is replaced by $\Unibar$ in step 1. Then
the modified scheme (called the ``moment
correction method'')  is the following:\\

\underline{\it Step 1:} Macroscopic quantities at $t_{n+1}$. 
\begin{equation}  \label{eq-Unpunibarbar}
\Unpuni=\Unibar -\frac{\Dt}{\Dx}\left(\Phi^n_{i+\demi}- \Phi^n_{i-\demi}\right),
\end{equation}
where the numerical fluxes are defined by
\begin{equation}  \label{eq-numfluxbar}
\Phi^n_{i+\demi}=\cintni{v^+m\fni} + \cintnipun{v^-m\fnipun} ,
\end{equation}

\underline{\it Step 2:} discrete velocity grid $\Vnpuni$ (step
unchanged).
\begin{equation}  \label{eq-boundsnpunibar}
\vminnpuni=\unpuni-4\sqrt{R\Tnpuni} \quad \text{ and } \quad 
\vmaxnpuni=\unpuni+4\sqrt{R\Tnpuni}.
\end{equation}
\begin{equation}\label{eq-Vnpunibar} 
\begin{split}
 \Vnpuni=\lbrace\vnpunik & =  \vminnpuni+(k-1)\Dvnpuni, 
\quad
k \text{ from } 1 \text{ to } \K  \rbrace,\\
& \text{ with }\Dvnpuni=(\vmaxnpuni-\vminnpuni)/(\K-1). 
\end{split}
\end{equation}

\bigskip
\underline{\it Step 3:} distribution function at time $t_{n+1}$ (step
unchanged).
\begin{equation}  \label{eq-fnpunikbar}
\begin{split}
  \fnpunik=\fnibar(\vnpunik) & -
  \frac{\Dt}{\Dx}{\vnpunik}^+(\fnibar(\vnpunik)-\fnimunbar(\vnpunik))
  - \frac{\Dt}{\Dx}{\vnpunik}^-(\fnipunbar(\vnpunik)-\fnibar(\vnpunik))  \\
  & + \frac{\Dt}{\taunpuni}(M_k(\Unpuni)-\fnpunik),
\end{split}
\end{equation}
for $k=1$ to $\K$. \\

\underline{\it Step 4:} Moment correction step.
\begin{equation}  \label{eq-Unpunibarmod}
  \Unpunibar=\cintnpuni{m\fnpuni}=\sum_{k=1}^\K m(\vnpunik) \fnpunik \omega_k.
\end{equation}

This means that the macroscopic quantities at time $t_{n+1}$ are modified to be the
moments of $\fnpuni$, and that the discrete conservation laws at the
next time step are initialized with these moments. This is similar to
a technique used in the ``moment guided'' method proposed in~\cite{DDP_2011}. 

For this modified scheme, the non-negativeness property~\ref{prop:pos} is still true, but
unfortunately, the conservation property~\ref{prop:cons} is lost: 
\modif{Indeed, even if we deduce from the discrete conservation laws~(\ref{eq-Unpunibarbar}) that 
\begin{equation*}
  \sum_iU^{n+1}_i \Delta x  = \sum_i U_i^{n,\star}\Delta x,
\end{equation*}
there is no way to link the corrected moment $U_i^{n,\star}$ to the
moment vector $U_i^{n}$ obtained at the previous time step, for the
same reason as given at the beginning of this section. Indeed, it is
likely that $U_i^{n,\star}$ is different from $U_i^{n}$, even if they
approximate the same value. Consequently, the scheme is not
conservative anymore.  }

\label{review_page1}Finally, we conclude this section by another modification that ensures
the positivity of the temperature $\Tnpuni$ in the previous modified
scheme. We propose to replace the quadratures used to compute the
macroscopic vector $\Unibar$ and the numerical flux $\Phi^n_{i+\demi}$
(see~(\ref{eq-Unpunibarmod}) and~(\ref{eq-numfluxbar})) by the exact integral of
the corresponding reconstructed functions, that is to say:
\begin{itemize}
  \item in step 1, (\ref{eq-numfluxbar}) is replaced by 
\begin{equation*}
  \Phi^n_{i+\demi}=\cint{v^+m\fnibar} + \cint{v^-m\fnipunbar},
\end{equation*}
\item in step 4, (\ref{eq-Unpunibarmod})  is replaced by
\begin{equation*}
   \Unibar=\cint{m\fnibar},
\end{equation*}
\end{itemize}
where we remind that $\cint{\phi}=\int_\R \phi(v)\, dv$ for
any velocity dependent function.
If the reconstruction procedure uses a polynomial interpolation, these
integrals are just integrals of piecewise polynomial functions
and can be evaluated explicitly. 
With this definition, the discrete conservation
law~(\ref{eq-Unpunibarbar}) reads
\begin{equation*}
  \Unpuni=\cint{m\Bigl( (1-\frac{\Dt}{\Dx}|v|)\fnibar +
    \frac{\Dt}{\Dx}v^+\fnimunbar - \frac{\Dt}{\Dx}v^-\fnipunbar
    \Bigr)} = \cint{m\phi}
\end{equation*}
where $\phi$ is a piecewise continuous function of $v$. Now we have the
following property.
\begin{prop} \label{prop:Urea}
  Under the CFL condition $\Dt\leq
  \Dx/\max_{i,k}(|\vnik|)$, the function $\phi$ is non negative, and
  hence $\Tnpuni$ is positive.
\end{prop}
\begin{proof}
  Observe that $\phi$ is a linear combination of $\fnibar$, $\fnimunbar$,
  $\fnipunbar$. The last two coefficients are always non negative. As
  a consequence of the CFL condition, the first coefficient, which is
  $1-\frac{\Dt}{\Dx}|v|$, is non negative if $|v|$ is small enough (that
  is to say $|v|\leq \max_{i,k}(|\vnik|)$). If $v$ is larger, the
  coefficient is negative, but by construction
  $\fnibar(v)=0$ (see~(\ref{eq-deffnibar})). Consequently, $\phi$ is non negative for every $v$, and
  hence $\Unpuni$ is realized by a non negative distribution. It is
  then a standard result that the corresponding density, energy and
  temperature are positive.
\end{proof}

However, we observe that in practice, the first modified
scheme~(\ref{eq-Unpunibarbar}--\ref{eq-Unpunibarmod}) preserves the
positivity. This is why we do not use this second modified scheme in
the numerical tests presented in this paper.

   \subsection{Reconstruction: from $\fnik$ to $\fnibar$}
 \label{interpol}

 To compute $\fnibar
  (v_{i,k}^{n+1})$ in equation (\ref{eq-fnpunik}), we have to use a
  reconstruction procedure. First, if $\vnpunik$ is external to $\Vni$, we set
  $\fnibar(\vnpunik)=0$:\label{review-page5} indeed, if the grid $\Vni$ is large
    enough, the distribution is very very small outside the grid, and
    it is reasonable to set it to 0. If $\vnpunik$ is inside $\Vni$, it is not a node of
  $\Vni$ in general, and we use a piecewise
  polynomial interpolation. We observed that first order polynomial
  interpolation is not accurate enough (a lot of discrete
    velocities are needed to get correct results). However, higher order
  interpolation produces oscillations, especially in very rarefied
  regimes, which is probably due to the large velocity gradients (and
  even discontinuities) of the distribution functions in such regimes.

Consequently, we use the essentially non oscillatory (ENO)
interpolation (see~\cite{shu_rapport_eno} or~\cite{FST2013}): with 3 or 4
point interpolation, the accuracy is good and there is almost no
oscillation.

The reconstruction algorithm is summarized below.
\begin{enumerate}
  \item If $|\vnpunik|>\max(|\vminni|,|\vmaxni|)$, then set
    $\fnibar(\vnpunik)=0$
  \item else 
\begin{enumerate}
  \item find the interval $[v^n_{i,k'},v^n_{i,k'+1}]$ inside $\Vni$
    that contains $\vnpunik$
  \item compute the $q$ point ENO polynomial function $P$ defined on the
    stencil $\lbrace v^n_{i,k'-(q-1)},v^n_{i,k'+q}\rbrace$
  \item set $\fnibar(\vnpunik)=P(\vnpunik)$
\end{enumerate}
\end{enumerate}

Of course, the same procedure is applied to the other reconstructed
values $\fnipunbar(\vnpunik)$ and $\fnimunbar(\vnpunik)$ in~(\ref{eq-fnpunik}).

\label{review-page6}\begin{remark}
  If we use a first order interpolation (linear reconstruction), the
positivity of the distribution function is preserved, this can be
proved easily. But, for the higher order ENO interpolation that we use
in practice, there is no reason why the positivity should always be
preserved. It could be interesting to look for modified interpolations
that preserve positivity. However in all the test cases that have
been studied here, this drawback does not induces a loss of positivity
of $\fnpuni$ in~(\ref{eq-fnpunik}).
\end{remark}

   \subsection{Non symmetric local discrete velocity grids}
  \label{subsec:nonsymmetric}

When the flow is far from equilibrium, the distribution functions are
different from their local Maxwellian, and might have a non symmetric
shape. In particular, their support might be non symmetric as well
(see for instance the heat transfer problem in the rarefied regime as
shown in section~\ref{subsec:heat}). However, the local grids defined
in section~\ref{subsec:ldv} are based on the local Maxwellians and are
necessarily symmetric. In this section, we propose a method to modify
the grid if necessary. \modif{This method can be applied to both versions
~(\ref{eq-Unpunibar}--\ref{eq-fnpunik})
and~(\ref{eq-Unpunibarbar}--\ref{eq-Unpunibarmod}) of our scheme.}

First, note that up to now, we have defined uniform local grids with a
constant number of points. However, the method is readily extended to
non uniform grids with a variable number of points : we just have to
replace $K$ by $K^n_i$ ad $\omega_k$ by $\omega^n_{i,k}$ in every
expressions of section~\ref{subsec:ldv}.

Then, the idea is to enlarge the grid $\Vnpuni$ if $\fnpunik$ is not
small enough at its boundaries. This is made by using a splitting
between the relaxation step and the transport step: we first compute
the intermediate quantity $ f_{i,k}^{n+1/2}$ by using the transport
equation as
\begin{equation}  \label{eq-transport}
\begin{split}
   f_{i,k}^{n+ 1/2} =\fnibar(\vnpunik)  - \frac{\Dt}{\Dx}{\vnpunik}^+(\fnibar(\vnpunik)-\fnimunbar(\vnpunik)) 
- \frac{\Dt}{\Dx}{\vnpunik}^-(\fnipunbar(\vnpunik)-\fnibar(\vnpunik)) ,
\end{split}
\end{equation}
for $k=1$ to $\K$. At the end of the transport step, the values of the
distribution function $f_{i}^{n+ 1/2}$ at the boundary points
$v^{n+1}_{i,max}$ and $v^{n+1}_{i,min}$ of $\Vnpuni$ are compared to
the maximum value of the distribution in the
grid: \label{review-page7} if the relative difference between
  one of these boundary values and the maximum value in the grid is
  larger than a tolerance (that was taken to $10^{-4}$ in our tests),
  then new grid points $w_R=v^{n+1}_{i,max}+\Dvnpuni$ or
  $w_L=v^{n+1}_{i,min}-\Dvnpuni$ are added outside the grid, and the
  corresponding values of $f_{i,k}^{n+ 1/2}$ are computed by
  using~(\ref{eq-transport}) again. This step is repeated until the
  left and right relative differences are smaller than the specified
  tolerance. At the end of this step, the modified grid $\Vnpuni$ now
has $K^{n+1}_i$ velocities. Finally, $ f_{i,k}^{n+1}$ is obtained from
the relaxation step through the relation
\begin{equation}  \label{eq-relax}
\begin{split}
  \fnpunik=   f_{i,k}^{n+ 1/2}     + \frac{\Dt}{\taunpuni}(M_k(\Unpuni)-\fnpunik),
\end{split}
\end{equation}
for $k=1$ to $K^{n+1}_i$. 

\modif{Note that that the method suggested here is just a modification
  ot the previous schemes. For
  scheme~(\ref{eq-Unpunibar}--\ref{eq-fnpunik}), the discrete kinetic
  equation~(\ref{eq-fnpunik}) is replaced by the transport/relaxation splitting
  (\ref{eq-transport}--\ref{eq-relax}). After the use of the transport step~(\ref{eq-transport}) on ${\cal
    V}_i^{n+1}$, it is used iteratively to add new points outside the
  grid, until the distribution is small enough, which leads to an
  ``enlarged'' grid, still denoted by ${\cal V}_i^{n+1}$. Then, the
  relaxation step~(\ref{eq-relax}) is
  used. Scheme~(\ref{eq-Unpunibarbar}--\ref{eq-Unpunibarmod}) can be
  modified accordingly.}

It would also be interesting to use an automatic refinement
of the local grid around the possible discontinuities, but this is not
studied in this paper. 

   \subsection{Extensions to other collision models}
   \label{subsec:esbgk}
This algorithm can be adapted to any collision model which is local in
space, like the ES-BGK or Shakhov models, or even the Boltzmann
collision operator itself. Indeed, as long as the grids are defined by
using the velocity and the temperature, we only have to use the
conservation laws (density, momentum, and energy), that are satisfied
by all the standard models. The fact that a model like ES-BGK contains
non-conservative quantities has no influence on the algorithm.

This is slightly different if one wants to use higher order moments to
define the local velocity grids. Indeed, we could imagine that the
shear stress and the heat flux, for instance, could also be used to define non
isotropic and non symmetric local velocity grids, even if this is not
what we advocate now. In that case, we would have to use
moment equations that are not conservation laws (evolution of the
pressure tensor and the energy flux). Then, the applicability of our
algorithm depends on the time approximation of the collision
operator. If we use an explicit time discretization, the right-hand
side of the higher-order moments equations can be explicitly computed
with the distributions at time $t_n$. However, if we use a
semi-implicit time discretization (as we do in this paper), the
discrete moment equations can be solved only if the right-hand side can
be written as a function of the moment vector at time $t_{n+1}$: this
is true for relaxation models like BGK er ES-BGK and Shakhov models,
but this is not true in general (for the Boltzmann collision operator
for instance).

However, note that this discussion makes sense only when the velocity
is at least of dimension 2 (since ES-BGK model and Boltzmann collision
operator do not exist for a one dimensional velocity), which will be
studied in a forthcoming work.

   \section{Numerical results}
\label{sec:res}

In this section we present three numerical tests to illustrate the
main features of our method (denoted by LDV, for local discrete
velocity grid). It is compared to a standard discrete velocity method
(with a global velocity grid) denoted by DVM
(see~\cite{luc_m3as}). First, the numerical scheme is tested on the
Sod test case for three different regimes: the rarefied, fluid, and
free transport regimes. The second test is the two interacting blast
waves problem in which very high temperature differences make the
standard DVM inefficient. The third test case is devoted to the heat
transfer problem, where the use of non symmetric local grids is shown
to be necessary when the Knudsen number increases. In these tests, the
gas constant $R$ is 208.1, except for the free transport regime in
section~\ref{transpfree} where $R=1$.

   \subsection{Sod test case}
   \label{subsec:sod}

\subsubsection{Rarefied regime}

\indent We consider a classical Sod test in a rarefied regime for the
BGK model (\ref{eq-bgk}) with the parameters $\omega = -0.19$ and $C =
1.08\, \cdot 10^{-9}$ used in the relaxation time $\tau$. The space
domain is the interval $[0,0.6]$ discretized with $300$ points. The
initial state is given by a local Maxwellian distribution whose
macroscopic quantities are
\begin{eqnarray} \label{sodinit}
\begin{split}
 & T(x) = 0.00480208, \; \rho(x) = 0.0001, \; u(x) =0 \; \mbox{for} \; x
\in [0,0.3] 
\\ & T(x) =
  0.00384167, \; \rho(x) = 0.0000125, \; u(x) =0 \;  \mbox{for} \; x
  \in ]0.3,0.6].
\end{split}
\end{eqnarray}

The DVM and LDV
 methods are compared to a reference solution given by the DVM method
 computed for a large and fine velocity grid (obtained after a
 convergence study). This
  velocity grid
  has $600$ points uniformly distributed in the interval $[ v_{min} , v_{max}]$ where
\begin{equation} \label{enlarged}
v_{min} =
  \min_{t,x}\bigl(u(t,x)-6\sqrt{RT(t,x)}\bigr), \hspace*{4 mm} v_{max} =
  \max_{t,x}\bigl(u(t,x)+6\sqrt{RT(t,x)}\bigr),
\end{equation}
which leads to bounds equal to $\pm 6$. 
Of course, such bounds cannot be determined a priori, and several
computations with larger and larger velocity grids have to be done
before the correct bounds are found. This illustrates the difficulty
to use a standard DVM when the extreme values of the velocity and the
temperature are not known a priori. Indeed, here, the temperature in
the shock after the initial time is higher than the two initial left
and right temperatures. If the velocity grid is computed with
formula~(\ref{eq-gbound}) and~(\ref{eq-gstep}) and the bounds are
estimated with the initial values of $T$ and $u$ 
(which gives bounds
equal to $\pm 4$)
, then the grid is not large
enough:  the results are not correct,  even if the number of
velocities is increased so as to reach the grid convergence, which is
obtained here with 100 velocities. This is shown in
figure~\ref{fig:sod_rarefied}.

At the contrary, our LDV method dynamically adapts to the time
variations of $u$ and $T$ and gives very accurate results with 30
discrete velocities only in each local grid, as it is shown in
figure~\ref{fig:sod_rarefied}. Consequently, the LDV
method is very efficient for this case. 

Note that in this test, the reconstruction procedure used the 4-points
ENO interpolation. See section~\ref{transpfree} for an analysis of the
influence of the order of interpolation.

\subsubsection{Fluid regime}
\label{subsubsec:fluid}

Now, we consider the same Sod test case, but in the fluid regime. This
regime corresponds to the limit case of equation (\ref{eq-bgk}) when
$\tau$ is set to $0$. Note that since both DVM and LDV methods are
used with a time semi-implicit scheme, taking $\tau=0$ means that
$f^{n+1}$ is set to $M(U^{n+1})$ at each time step (and hence the
choice of the interpolation procedure has no influence), and we get
two different numerical schemes for the compressible Euler equations
of gas dynamics.

Here, the reference DVM has 100 velocities only 
and bounds equal to $\pm 5.2$
. The DVM with the grid
computed with the initial values of $T$ and $u$ has also 100
velocities
with bounds equal to $\pm 4$
, but it still gives incorrect values (see
figure~\ref{fig:sod_fluid}), for the same reasons as mentioned in the
previous section. At the contrary, our LDV method gives very accurate
results with 10 velocities only. 

\label{review-page3} Note that the results obtained in this fluid
  regime are very close to those obtained in the rarefied regime (even
  if the shock profile is stiffer in the fluid regime, as
  expected). This can be understood by computing the Knudsen number of
  the rarefied regime. In this test, the initial mean free path is
  between $3 \, 10^{-5}$ (left state) and $2 \, 10 ^{-4}$ (right
  state). It is difficult to define a Knudsen number here, since there
  is no macroscopic length scale, but if we choose the length of the
  computational domain, we find a Knudsen number between $5 \,
  10^{-5}$ and $3\, 10^{-4}$, which is quite small.

\subsubsection{Free transport regime} 
\label{transpfree} 

We consider the free transport regime corresponding to the limit case
in (\ref{eq-bgk}) when $\tau$ tends to $+ \infty$. In this section,
we take $R=1$ and the standard dimensionless values for the Sod test
case. The space domain is $[-1,1]$ and is discretized with $300$
points. The distribution function is initialized by the local
Maxwellian distribution function with
\begin{eqnarray} \label{transpinit}
\begin{split} 
\rho=1, \;   u=0, \;   p=1 \quad  \mbox{in } [-1,0[, 
 \\   
\rho=0.125, \;  u=0 , \;  p=0.1  \quad  \mbox{in }  [0,1].
\end{split}
\end{eqnarray}
For this test case the numerical results are compared to the
analytical solution of the free transport equation. 

It is very difficult to accurately approximate the free transport
equation with a standard DVM: the macroscopic profiles obtained with
the DVM show several plateaux. These plateaux are not due to the space
and time approximation, but are only due to the velocity
discretization. Indeed, it can be easily proved that the macroscopic
profiles of the DVM solution at time $t$ have plateaux of length
$t\Dv$, where $\Dv$ is the step of the global uniform grid. This is
clearly seen in figures~\ref{freesodtemperature}--\ref{freesoddensite}
(top), where the DVM has 30 discrete velocities
with bounds equal to $\pm 4$
. This phenomenon is
known as the ray effect that appears with the discrete ordinate
approximations of the radiative transfer equation.

When we test this problem with our LDV method, with 2 point piecewise linear
interpolation, and 30 velocities in each local grids, the results are very
bad: we observe very large oscillations (see
figures~\ref{freesodtemperature}--\ref{freesoddensite}, bottom). This is
probably due to the fact that this interpolation is not accurate
enough. Then we use 3 and 4 point ENO interpolations, and we observe in
the same figures that the solution is now much closer to the exact
solution. Moreover, while we have the same number of discrete
velocities in each local grid as in the DVM, we observe that the
plateaux are completely eliminated.

However, there are still some oscillations in the results obtained
with the LDV \modif{(note that these oscillations are not
amplified and that the numerical solution remains bounded for larger
times)}. The oscillations are probably due to the fact that these results were
obtained without the moment correction method. Indeed, if we do now
the same simulation with this moment correction method
(scheme~(\ref{eq-Unpunibarbar}--\ref{eq-Unpunibarmod})), the
results are very good: there are much less oscillations, almost no
plateau phenomenon, and the results are much closer to the analytical
solution, see
figures~\ref{fig:sod_step4_temperature}--\ref{fig:sod_step4_density}.

\label{review-page2} Up to now, we do not know the reason for these oscillations and
  it is not clear why they are eliminated when we use the moment
  correction method. Our intuition is that in the original method, there
  is some incompatibility between the discretization of the
  conservation laws and the discretization of the kinetic equation:
  namely, the moments of the discrete kinetic equation do not lead to
  the discrete conservation laws that are solved. This is why the
  discrete conservation laws and the discrete kinetic equation lead to
  two different approximations of the moment vector at time
  $t_{n+1}$. The moment correction method forces these quantities to
  be equal.

   \subsection{Two interacting blast waves}
   \label{subsec:blast}

   This section is devoted to the test case called ``the two
   interacting blast waves'' (see~\cite{woodward}). Here, the
   relaxation time is defined with $\omega = -0.19$ and $C = 1.08\,
   10^{-9}$. The space domain is the interval $[0,1]$ which is still
   discretized with $300$ points. The initial distribution function is
   a local Maxwellian distribution whose macroscopic quantities are
   given by $\rho  =1$ and $u =0$ everywhere, and 
\begin{equation*}
 \; T =4.8, \; \mbox{in}  \, [0, 0.1], \quad T = 4.8
\, 10^{-5} \; \mbox{in} \; ]0.1 , 0.9] , \quad  T =4.8 \, 10^{-1} \;
\mbox{in} \; ]0.9,1].
\end{equation*}
On the left and right boundaries, we use Neumann boundary conditions:
\label{review-page9} we set the values of $f$, $u$, and $T$ in boundary ghost-cells to
  their values in the corresponding real boundary cells.

Here, the bounds of the global grid of the DVM are determined by the
largest initial temperature (we get $\pm126.5$), and its
step size is given by the smallest initial temperature. Then we find
that the coarsest global grid that satisfies
conditions~(\ref{eq-gbound})-(\ref{eq-gstep}) has not less than $2\,
551$ velocities! In
figures~\ref{temperatureblast}--\ref{pressionblast}, we observe that
the LDV method requires only $30$ velocities to give results that are
very close to the DVM method (with $2\, 551$ velocities), both before
and after the shock. This proves the high efficiency of the LDV
approach for this case. Note that here again, we use a 4 point ENO
interpolation in the LDV method.

Finally, we plot in figure~\ref{fig:blast_grids} some local velocity grids
for different space positions: in these plots, each vertical line is a
local velocity grid, and its nodes are the small dots on the
line. Note that before the waves interaction
(top), the velocity grids in the middle of the domain are much smaller
than the grid in the left state, which is due to the different
order of magnitude of the temperature at this time. After the
interaction, the temperature is more homogeneous, and the grids as well
(see the bottom plot in this figure).

   \subsection{Heat transfer problem}
    \label{subsec:heat}
In this test, we consider the evolution of a gas enclosed between two
walls kept at temperature $T_L=300$ and $T_R=1000$. At these walls, the
distribution function satisfies the diffuse boundary condition
\begin{eqnarray} \label{maxwelltransfert}
f(x=0,v>0 ) = M_{\rho_L,0,T_L} , \hspace*{3 mm} 
f(x=1,v<0 ) = M_{\rho_R,0,T_R}
\end{eqnarray}
where
\begin{eqnarray} \label{sigma}
 \rho_L = - \frac{\int_{v <0} v f(x=0,v) dv}{\int_{v >0} v
   M_{1,0,T_L} dv}, \hspace*{3 mm}
  \rho_R = - \frac{\int_{v >0} v f(x=1,v) dv}{\int_{v <0} v
   M_{1,0,T_R} dv},
\end{eqnarray}
and $M_{\rho,0,T}$ denotes $\rho/\sqrt{2\pi RT}\exp(-v^2/2RT)$ for
every $\rho$ and $T$. The initial data is the Maxwellian with density
$\rho_0$ (to be defined later), velocity $u_0=0$, and
$T_0=300$. Here, the relaxation time is defined with $\omega =
-0.5$ and $c = 6.15\, \cdot 10^{-9}$. 

The boundary conditions are taken into account in our numerical scheme
by a ghost cell technique, as it is standard in finite volume
schemes. Left and right ghost cells are defined for $i=0$ and
$i=i_{max}+1$, and the velocity grids ${\cal V}^n_0$ and ${\cal V}^n_{i_{max}+1}$
in these cells are defined so as to correctly describe the
corresponding wall Maxwellians. Then, the density $\rho_L$ and
$\rho_R$, that are defined as the ratio of an outgoing mass flux at
a wall to the corresponding incoming Maxwellian mass flux, are
approximated by using the boundary cell and the corresponding ghost
cell, that is to say: 
\begin{eqnarray*}
\rho_L= - \frac{ \langle v^- f_1^n \rangle_{ {\cal V}^n_{1} }
}{ \langle v^+ M(1,0,T_L)
\rangle_{{\cal V}^n_{0} }}, \hspace*{3 mm}
  \rho_R = - \frac{ \langle v^+ f_{i_{max}}^n \rangle_{{\cal
        V}^n_{i_{max}}}}{ \langle  v^- M(1,0,T_R) \rangle_{{\cal V}^n_{i_{max}+1} } } .
\end{eqnarray*}

\modif{For this test, we use the moment correction method with fourth
  order ENO interpolation, except for the computations discussed at the end of
  this section (see remark~\ref{rem:P1} below).}  We also use several
Knudsen numbers $\Kn$ here.  This number is parametrized by the
initial density $\rho_0$. We first analyze the LDV method in the
transitional regime ($\rho_0=1.88862\ 10^{-5}$, which gives
$\Kn=10^{-2}$): here, both the LDV and DVM are converged with 30
velocities (the bounds of the global grid are $\pm1825.34$), while the
space domain $[0,1]$ is discretized with $1000$ points.  However, the
results obtained with the LDV are not accurate enough (see
figure~\ref{fig:heat_trans}). An analysis of this problem shows that
this is due to the local grid close to the right boundary which is not
large enough: for small times, the distributions at these points are
highly non symmetric.

To correct this problem, we use the algorithm proposed in
section~\ref{subsec:nonsymmetric} to enlarge the local grids in a non
symmetric way. Then, the LDV with 30 velocities now gives results that
are indistinguishable from the DVM, see figure~\ref{transfertelargievitesse_etape4}. 

Then, we test the LDV method in the rarefied regime ($\Kn=1$,
$\rho_0=1.88862\ 10^{-7}$): the
LDV (with enlarged non symmetric local grids) and the DVM are converged with 300
velocities, while the space domain $[0,1]$ is discretized with $300$
points. Here again, both methods give results that are almost
indistinguishable (see figure~\ref{fig:heat_kn0}). Unfortunately, the
number of velocities required to get converged results is very large
here, even for the LDV (300 velocities). This is probably due to the
fact that the distribution function is discontinuous with respect to
the velocity in this test, with very large jumps: our velocity grids,
even the local ones, are uniform, and cannot capture these
discontinuities when the number of velocities is too small. However, we
show in figure~\ref{fig:heat_kn0_50} the results obtained with 50
velocities only, and we observe that the LDV gives results that are
more accurate than the DVM. 

\modif{Finally, we did the same computations for $\Kn=10$
  ($\rho_0=1.88862\ 10^{-8}$) and $\Kn=1000$ ($\rho_0 = 1.88862\
  10^{-10}$) and we observed the following. First, the plateau
  phenomenon in the DVM approach with 100 velocities is clearly seen
  with Kn$=10$ (while it was only slightly visible for $\Kn=1$). For
  $\Kn=1000$, the results are almost the same.  For the LDV, when the
  number of discrete velocities is the same as for the DVM, the
  results are close to the reference solution (no plateau, no
  oscillations) for $\Kn=10$ and $\Kn=1000$, which is much better than
  with the DVM (see
  figures~\ref{fig:heat_kn10} and~\ref{fig:heat_kn1000}). When the number of
  velocities is not large enough (30 points tested here), both method
  give wrong results, even if the LDV is better.}

 \begin{remark}\label{rem:P1}

\modif{We also compare our LDV method with and without the moment correction
step, for P1 and fourth order ENO interpolation. We do not find it usefull to add
the corresponding curves in this paper, but our observations are
summarized below:
\begin{itemize}

  \item for $\Kn=0.01$, there is not much
difference if the moment correction method is used or not:
\begin{itemize}
  \item with fourth order ENO interpolation, both methods
give good results, even if we note a solution which is slightly less
smooth without the moment correction step (there are a
few small peaks). 
\item if the P1 interpolation is used, both methods give a wrong
  solution (with small oscillations without the moment correction method). 
\end{itemize}

\item for larger Kn, ($\Kn=1, 10, 1000$): 
\begin{itemize}
  \item if the moment correction method is
not used, we observe oscillations whose amplitude increases with Kn,
regardless of the interpolation.
\item with the moment correction method and the P1 interpolation, the
  difference between the numerical results and the reference solution
  increases with Kn, except if the number of discrete velocities is
  increased too.
\end{itemize}
\end{itemize}
}

 \end{remark}

   \subsection{CPU time comparisons}
   We have compared the CPU cost of our method (on a single processor
   Pentium(R) Dual-Core CPU E6500\@2.9GHz) to the standard DVM with a
   global grid, by using the fortran subroutine {\tt cpu\_time}, and the following test cases:
\begin{itemize}
\item Sod test: we compared the DVM with 30 discrete velocities (shown
  in figures~\ref{freesodtemperature}--\ref{freesoddensite} (top) to
  our LDV with 30 points, 4 points ENO interpolation, and the moment
  correction method (shown in
  figures~\ref{fig:sod_step4_temperature}--\ref{fig:sod_step4_density});
 \item Blast waves: test case shown in
   figures~\ref{temperatureblast}--\ref{pressionblast} (bottom);
 \item Heat transfer: comparison between the LDV with 50 velocities
   and non symmetric grids, to the DVM with 50 velocities (see
   figure~\ref{fig:heat_kn0_50}).
\end{itemize}

We observed (see table~\ref{table:cpu}) that our method with local
velocity grids is more expensive than the global grid method for Sod
et heat transfer tests. Since the number of velocities in the grids is
smaller with our method, this increase in CPU cost is probably due to
the very large number of interpolations made in the evaluation of
reconstructed distributions. At the contrary, for the blast wave
problem, the number of velocities is so large in the global grid that
our method is less expensive. However, as expected, the gain factor in
CPU time (which is 45) is smaller than the gain factor in number of
velocities (which is 85).

We point out that the implementation of our method has not been
optimized in this work, and such an optimization would probably make
the method faster. However, these comparisons show that our algorithms
have to be improved to be less computationally expensive, in
particular to reduce the cost of the reconstructions of the
distribution in their local grids. This is now investigated in a forthcoming
work for 2D problems.

\section{Conclusion and perspectives}

We have presented a new velocity discretization of kinetic equations
of rarefied gas dynamics: in this method, the distribution functions
are discretized with velocity grids that are local in time and space,
contrary to standard discrete velocity or discrete ordinate
methods. The local grids dynamically adapt to time and space
variations of the support of the distribution function, by using the
conservation laws. 

This method is very efficient in case of strong variations of the
temperature, for which a standard discrete velocity method requires a
very large number of velocities. Moreover, it allows to eliminate the
plateau phenomenon in very rarefied regimes. 

\label{review-page4} We mention that in this study, the space
  discretization is a simple first order upwind method, which is known
  to have a very low accuracy. However, our method is quite
  independent of the space approximation: any higher order finite
  volume or finite difference approximation could be used. For
  instance, a second order upwind scheme with limiters can be used
  very easily by adding a flux limiter in (8) and slope limiters in
  (12).  However, in this preliminary work, we find it simpler to use
  a first order scheme to analyze the properties of the method, and to
  compare its advantages and drawbacks. We defer the investigation of
  higher order schemes to a work in progress in which our method will
  be extended to multi-dimensional problems and its computational cost
  will be reduced. 

Moreover, several aspects of the method should be better understood, in
particular, why are there some oscillations if the moment correction
method is not used, in the rarefied and free transport regimes, even
with high order ENO interpolation? A mathematical analysis of the
numerical method could be interesting here.

\paragraph{ Acknowledgements.} 
 This study has been carried out in the frame of “the Investments for the future” Programme IdEx Bordeaux – CPU (ANR-10-IDEX-03-02).

 \bibliographystyle{plain}

 \bibliography{biblio}

\begin{thebibliography}{10}

\bibitem{Aristov77}
V.V. Aristov.
\newblock Method of adaptative meshes in velocity space for the intense shock
  wave problem.
\newblock {\em USSR J. Comput. Math. Math. Phys.}, 17(4):261--267, 1977.

\bibitem{BCHM_saragosse}
C.~Baranger, J.~Claudel, N.~Herouard, and L.~Mieussens.
\newblock Locally refined discrete velocity grids for deterministic rarefied
  flow simulations.
\newblock {\em AIP Conference Proceedings}, 1501(1):389--396, 2012.

\bibitem{BCHM_2012}
C.~Baranger, J.~Claudel, N.~Hérouard, and L.~Mieussens.
\newblock Locally refined discrete velocity grids for stationary rarefied flow
  simulations.
\newblock {\em Journal of Computational Physics}, 257, Part A(0):572 -- 593,
  2014.

\bibitem{BLM_2008}
Mounir Bennoune, Mohammed Lemou, and Luc Mieussens.
\newblock Uniformly stable numerical schemes for the boltzmann equation
  preserving the compressible navier–stokes asymptotics.
\newblock {\em Journal of Computational Physics}, 227(8):3781 -- 3803, 2008.

\bibitem{bird}
G.A. Bird.
\newblock {\em {M}olecular {G}as {D}ynamics and the {D}irect {S}imulation of
  {G}as {F}lows}.
\newblock Oxford Science Publications, 1994.

\bibitem{BR_1999}
A.~V. Bobylev and S.~Rjasanow.
\newblock Fast deterministic method of solving the {B}oltzmann equation for
  hard spheres.
\newblock {\em Eur. J. Mech. B Fluids}, 18(5):869--887, 1999.

\bibitem{BPS}
A.V. Bobylev, A.~Palczewski, and J.~Schneider.
\newblock {A} {C}onsistency {R}esult for a {D}iscrete-{V}elocity {M}odel of the
  {B}oltzmann {E}quation.
\newblock {\em Siam J. Numer. Anal.}, 34(5):1865--1883, 1997.

\bibitem{buet_ttsp}
C.~Buet.
\newblock {A} {D}iscrete-{V}elocity {S}cheme for the {B}oltzmann {O}perator of
  {R}arefied {G}as {D}ynamics.
\newblock {\em Transp. Th. Stat. Phys.}, 25(1):33--60, 1996.

\bibitem{CXLC2012}
S.~Chen, K.~Xu, C.~Lee, and Q.~Cai.
\newblock A unified gas kinetic scheme with moving mesh and velocity space
  adaptation.
\newblock {\em Journal of Computational Physics}, 231(20):6643 -- 6664, 2012.

\bibitem{DDP_2011}
Pierre Degond, Giacomo Dimarco, and Lorenzo Pareschi.
\newblock The moment-guided {M}onte {C}arlo method.
\newblock {\em Internat. J. Numer. Methods Fluids}, 67(2):189--213, 2011.

\bibitem{DP_2013}
G.~Dimarco and L.~Pareschi.
\newblock Asymptotic preserving implicit-explicit {R}unge-{K}utta methods for
  nonlinear kinetic equations.
\newblock {\em SIAM J. Numer. Anal.}, 51(2):1064--1087, 2013.

\bibitem{FR2013}
F.~Filbet and T.~Rey.
\newblock {A Rescaling Velocity Method for Dissipative Kinetic Equations -
  Applications to Granular Media}.
\newblock 27 pages, 2012.

\bibitem{FST2013}
U.~Fjordholm, S.~Mishra, and E.~Tadmor.
\newblock Eno reconstruction and eno interpolation are stable.
\newblock {\em Foundations of Computational Mathematics}, 13:139--159, 2013.

\bibitem{HG2012}
J.R.Haack and I.M. Gamba.
\newblock Conservative deterministic spectral boltzmann solver near the grazing
  collisions limit.
\newblock {\em AIP Conference Proceedings}, 2012.

\bibitem{KA2012}
V.I. Kolobov and R.R. Arslanbekov.
\newblock Towards adaptive kinetic-fluid simulations of weakly ionized plasmas.
\newblock {\em Journal of Computational Physics}, 231(3):839 -- 869, 2012.

\bibitem{KAAFZ}
V.I. Kolobov, R.R. Arslanbekov, V.V. Aristov, A.A. Frolova, and S.A. Zabelok.
\newblock Unified solver for rarefied and continuum flows with adaptive mesh
  and algorithm refinement.
\newblock {\em Journal of Computational Physics}, 223(2):589 -- 608, 2007.

\bibitem{KAF2011}
V.I. Kolobov, R.R. Arslanbekov, and A.A. Frolova.
\newblock Boltzmann solver with adaptive mesh in velocity space.
\newblock In {\em 27th International Symposium on Rarefied Gas Dynamics},
  volume 133 of {\em AIP Conf. Proc.}, pages 928--933. AIP, 2011.

\bibitem{luc_m3as}
L.~Mieussens.
\newblock {D}iscrete {V}elocity {M}odel and {I}mplicit {S}cheme for the {BGK}
  {E}quation of {R}arefied {G}as {D}ynamics.
\newblock {\em Math. Models and Meth. in Appl. Sci.}, 8(10):1121--1149, 2000.

\bibitem{luc_jcp}
L.~Mieussens.
\newblock {D}iscrete-velocity models and numerical schemes for the
  {B}oltzmann-{BGK} equation in plane and axisymmetric geometries.
\newblock {\em J. Comput. Phys.}, 162:429--466, 2000.

\bibitem{PH2002}
V.A. Panferov and A.~G. Heintz.
\newblock A new consistent discrete-velocity model for the boltzmann equation.
\newblock {\em Mathematical Methods in the Applied Sciences}, 25(7):571--593,
  2002.

\bibitem{PP2007}
Sandra Pieraccini and Gabriella Puppo.
\newblock Implicit–explicit schemes for bgk kinetic equations.
\newblock {\em Journal of Scientific Computing}, 32:1--28, 2007.

\bibitem{woodward}
P.Woodward and P.Colella.
\newblock The numerical simulation of two-dimensional fluid flow with strong
  shocks.
\newblock {\em J. Comput. Phys.}, 54:115--173, 1984.

\bibitem{rogier-schneider}
F.~Rogier and J.~Schneider.
\newblock {A} {D}irect {M}ethod {F}or {S}olving the {B}oltzmann {E}quation.
\newblock {\em Transp. Th. Stat. Phys.}, 23(1-3):313--338, 1994.

\bibitem{shu_rapport_eno}
C.-W. Shu.
\newblock Essentially non-oscillatory and weighted essentially non-oscillatory
  schemes for hyperbolic conservation laws.
\newblock Technical Report 97-65, ICASE, 1997.

\bibitem{TSA1993}
S.~Takata, Y.~Sone, and K.~Aoki.
\newblock Numerical analysis of a uniform flow of a rarefied gas past a sphere
  on the basis of the boltzmann equation for hard-sphere molecules.
\newblock {\em Physics of Fluids A: Fluid Dynamics}, 5(3):716--737, 1993.

\bibitem{T_cicp}
V.~A. Titarev.
\newblock Efficient deterministic modelling of three-dimensional rarefied gas
  flows.
\newblock {\em Communications in Computational Physics}, 12(1):162--192, 2012.

\bibitem{XH2010}
K.~Xu and J.-C. Huang.
\newblock A unified gas-kinetic scheme for continuum and rarefied flows.
\newblock {\em J. Comput. Phys.}, 229:7747--7764, 2010.

\end{thebibliography}

\clearpage
\begin{figure}
  \centering
\includegraphics[height=0.3\textheight]{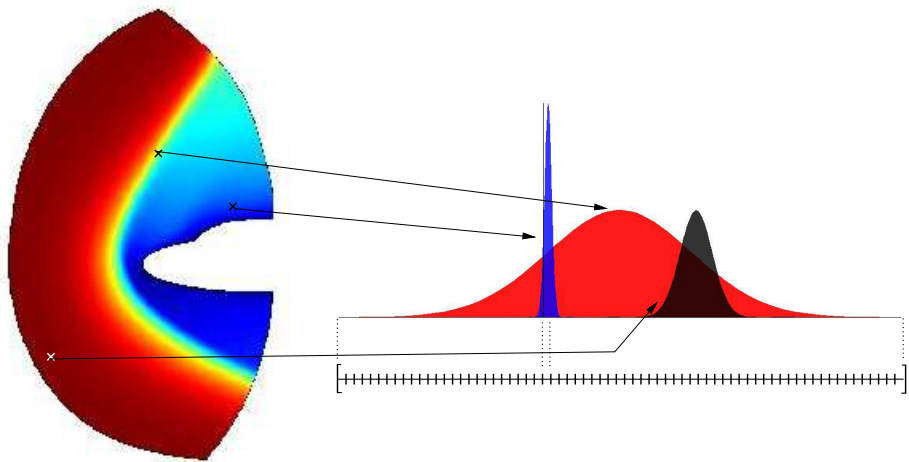}    
\caption{Three distribution functions in different space points of a
  computational domain for a re-entry problem, and the corresponding
  global discrete velocity grid ${\cal V}$.}
  \label{fig:hermes}
\end{figure}

\clearpage
\begin{figure}[!ht]
\begin{center}
\begin{minipage}[c]{0.75\textwidth}
\subfigure{
\includegraphics[angle=-90, width=\textwidth]{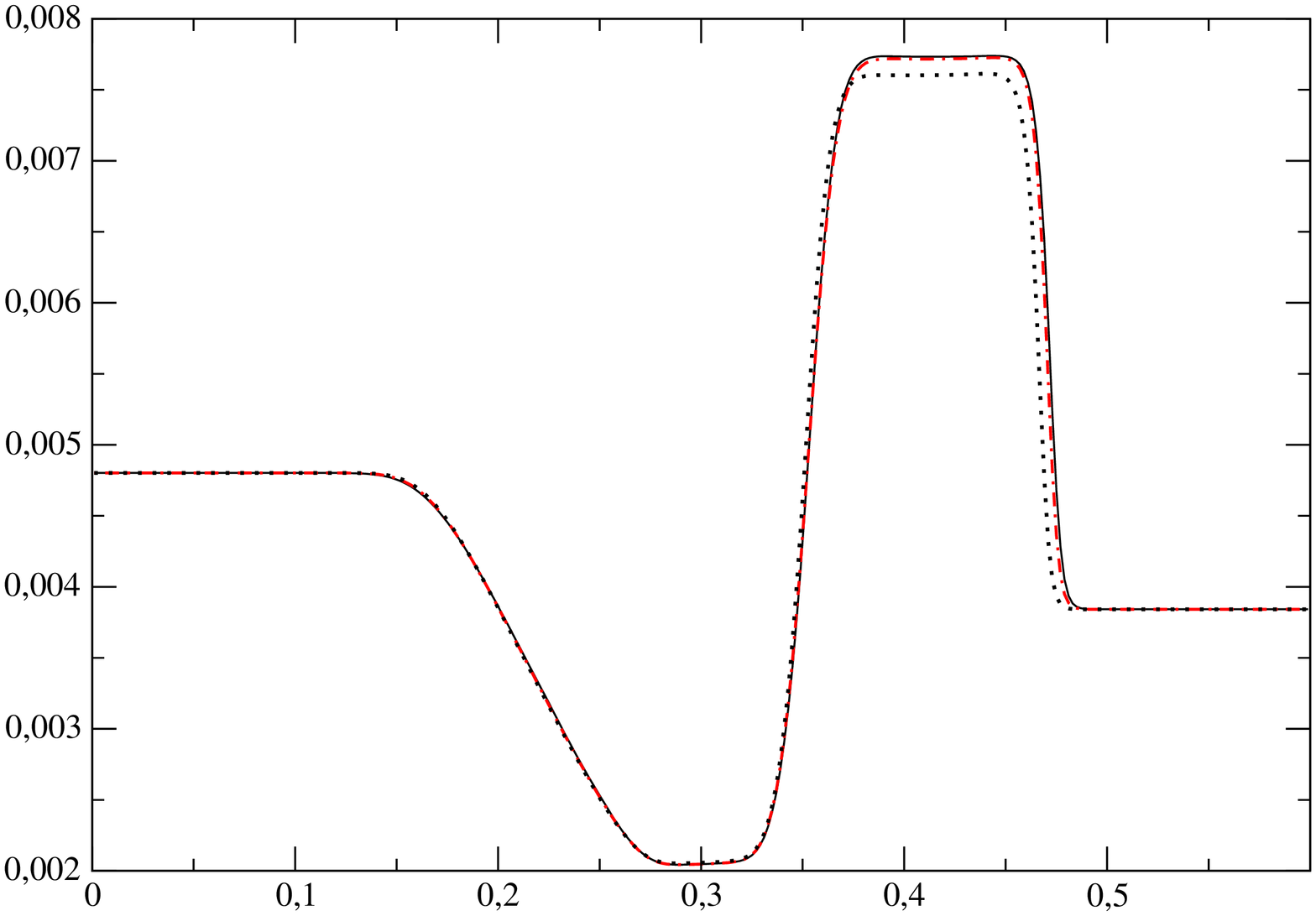}}
\end{minipage}\hfill%
\begin{minipage}[c]{0.77\textwidth}
\subfigure{
\includegraphics[angle = -90, width=\textwidth]{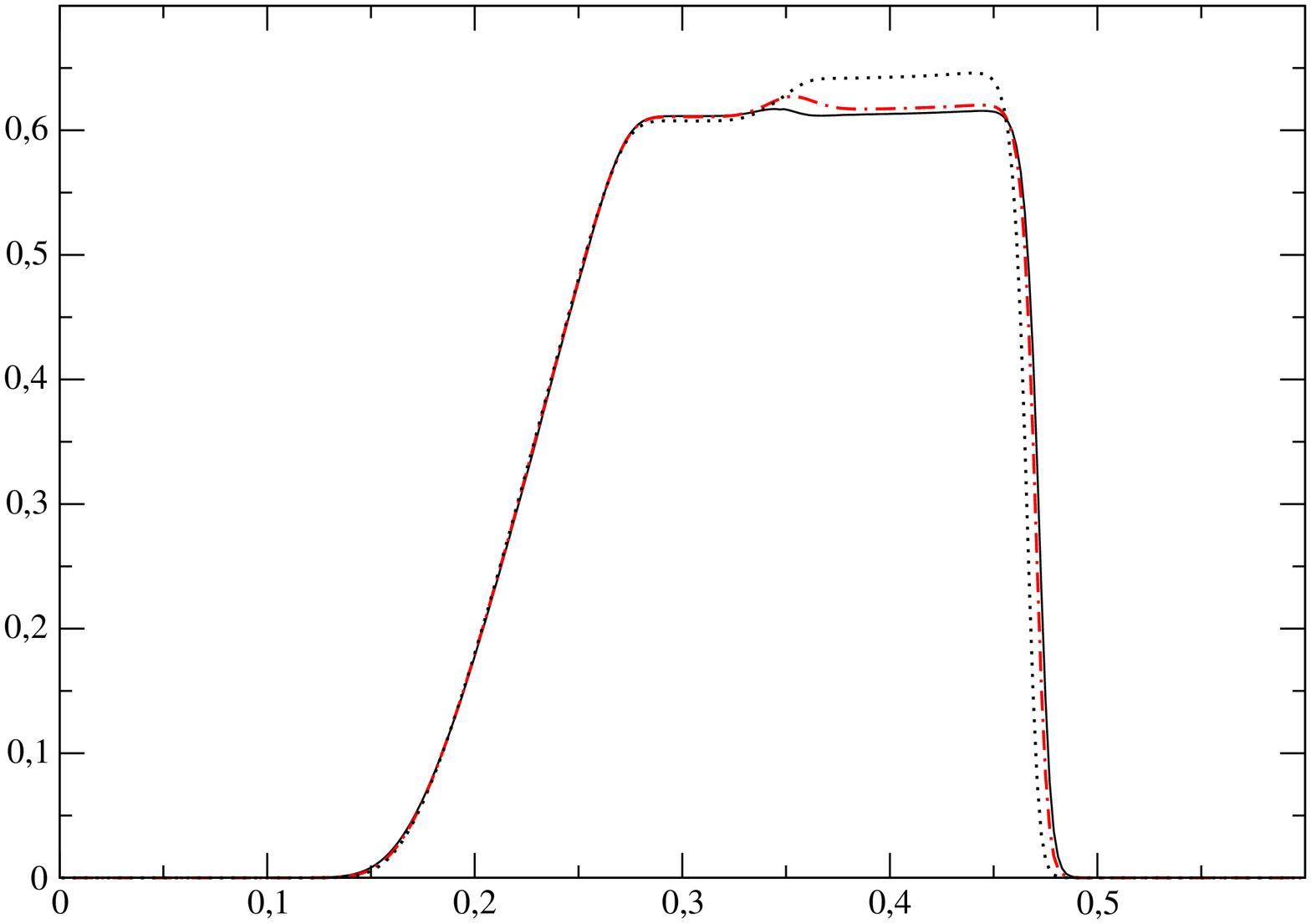}}
\end{minipage} 
\caption{Sod test case, rarefied regime: temperature (top) and
  velocity (bottom) profiles, at time
  $7.34 \; 10^{-2}$. The solid line is the reference solution obtained
  with the DVM with the enlarged
  global grid (600 points), the dotted line is the DVM with an
  incorrect grid (100 points), while the dot-dashed line is the LDV
  method (30 points).\label{fig:sod_rarefied}}
\end{center}
\end{figure}

\clearpage
\begin{figure}[!ht]
\begin{center}
\begin{minipage}[c]{0.75\textwidth}
\subfigure{
\includegraphics[angle=-90, width=\textwidth]{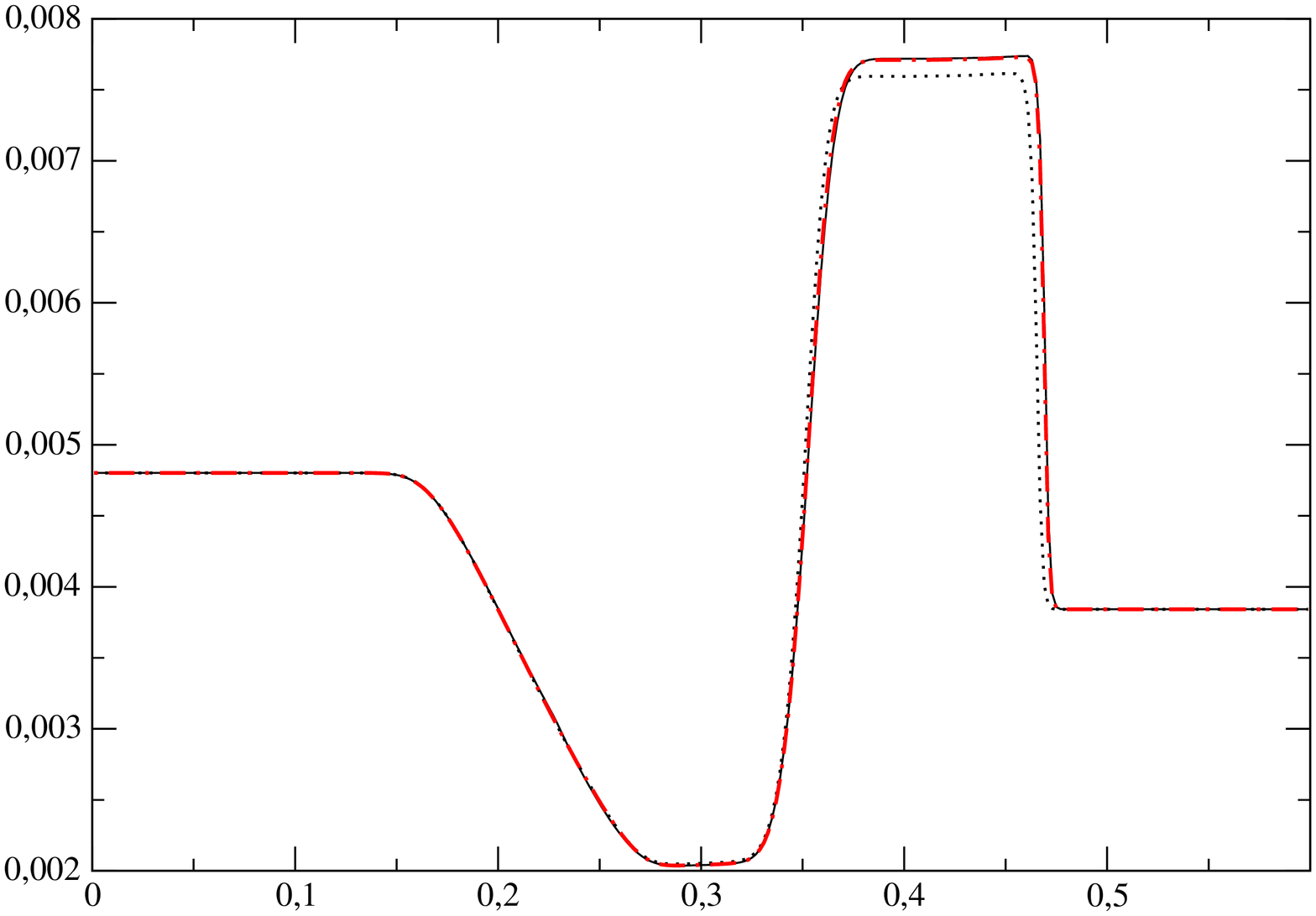}}
\end{minipage}
\begin{minipage}[c]{0.77\textwidth}
\subfigure{
\includegraphics[angle=-90, width=\textwidth]{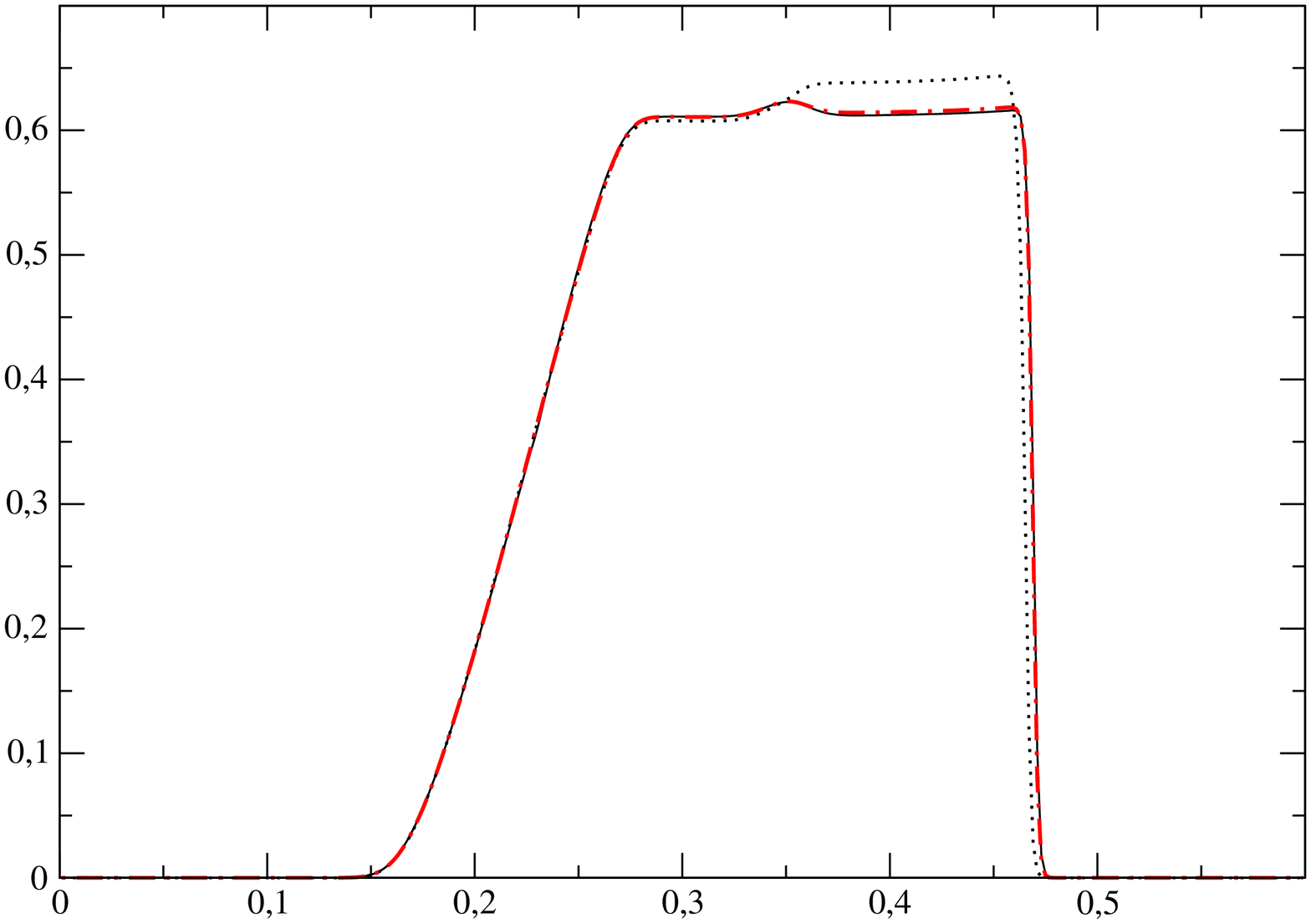}}
\end{minipage} 
\caption{Sod test case, fluid regime: temperature (top) and
  velocity (bottom) profiles, at time
  $7.34 \; 10^{-2}$. The solid line is the reference solution obtained
  with the DVM with the enlarged
  global grid (100 points), the dotted line is the DVM with an
  incorrect grid (100 points), while the dot-dashed line is the LDV
  method (10 points).\label{fig:sod_fluid}}
\end{center}
\end{figure}

\clearpage
\begin{figure}[!ht]
\begin{center}
\begin{minipage}[c]{0.73\textwidth}
\subfigure{
\includegraphics[
  angle=-90,width=\textwidth]{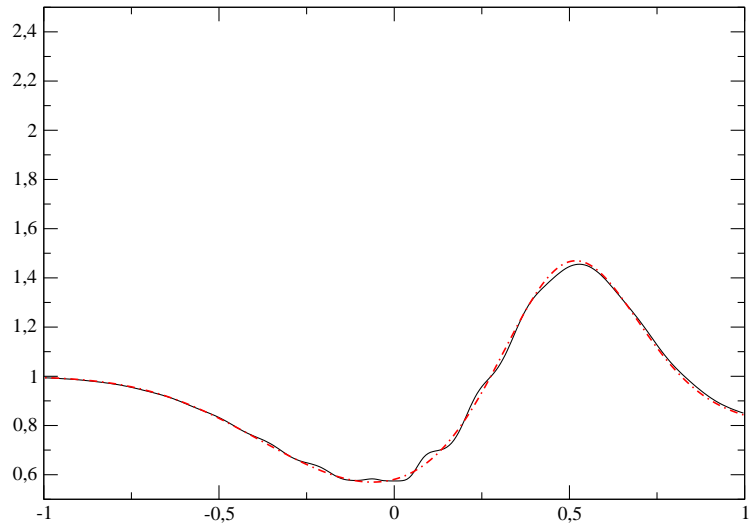}}
\end{minipage} 
\begin{minipage}[c]{0.73\textwidth}
\subfigure{
\includegraphics[ angle=-90, width=\textwidth]{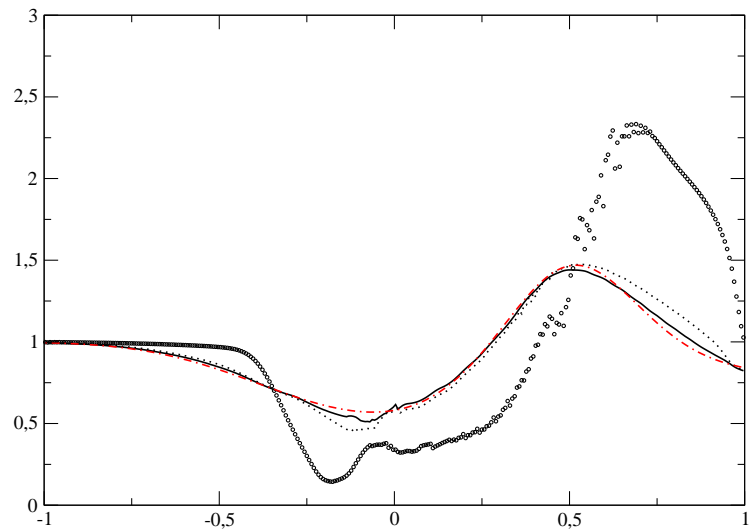}}
\end{minipage}
\caption{Sod test case, free transport: temperature at
  time $0.3$. Top: comparison between the exact solution (dot-dashed)
  and with a 30 points DVM (solid). Bottom, comparison between
  the exact solution (dot-dashed) and several LDVs with 30 points:
  with first order interpolation ('o'), with 3 points-ENO (dotted), with 4
  points-ENO (solid). } \label{freesodtemperature}
\end{center}
\end{figure}

\clearpage
\begin{figure}
\begin{center}
\begin{minipage}[c]{0.77\textwidth}
\subfigure{
\includegraphics[angle=-90,
  width=\textwidth]{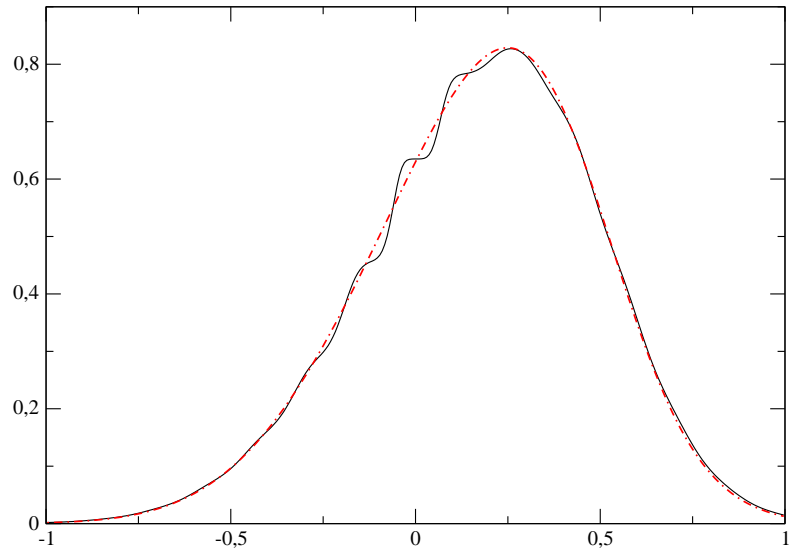}}
\end{minipage} 
\begin{minipage}[c]{0.77\textwidth}
\subfigure{
\includegraphics[angle = -90, width=\textwidth]{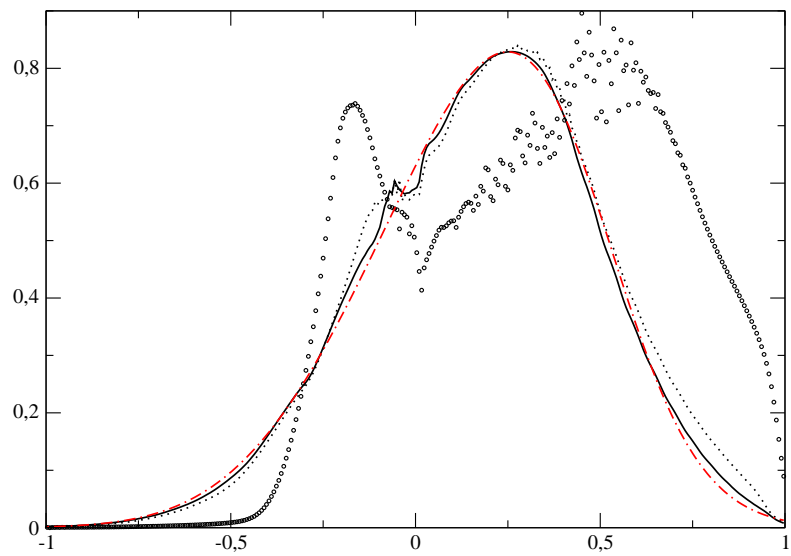}}
\end{minipage}
\caption{Same as figure \ref{freesodtemperature} for the velocity.} \label{freesodvitesse}
\end{center}
\end{figure}

\clearpage
\begin{figure}
\begin{center}
\begin{minipage}[c]{0.77\textwidth}
\subfigure{
\includegraphics[angle=-90, width=\textwidth]{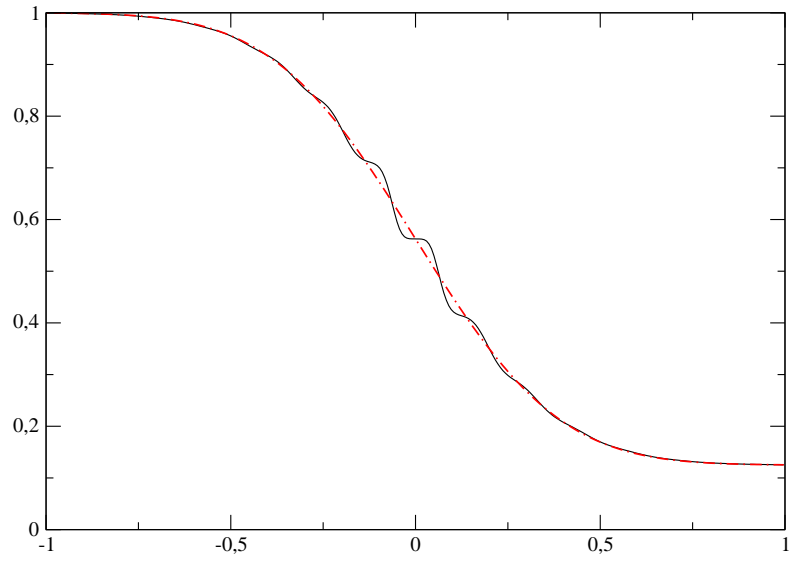}}
\end{minipage}
\begin{minipage}[c]{0.77\textwidth}
\subfigure{
\includegraphics[angle=-90, width=\textwidth]{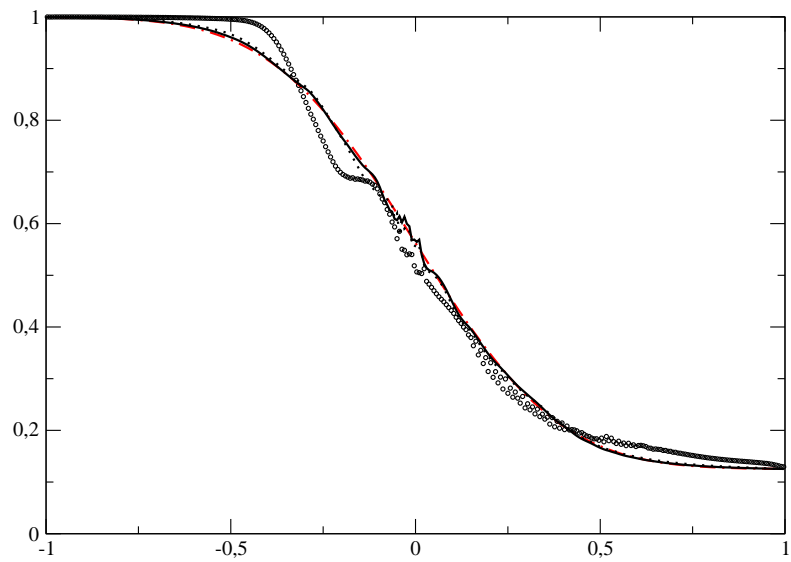}}
\end{minipage}
\caption{Same as figure \ref{freesodtemperature} for the density.} \label{freesoddensite}
\end{center}
\end{figure}

\clearpage
\begin{figure}[!ht]
\begin{center}
\includegraphics[angle=-90,
width=\textwidth]{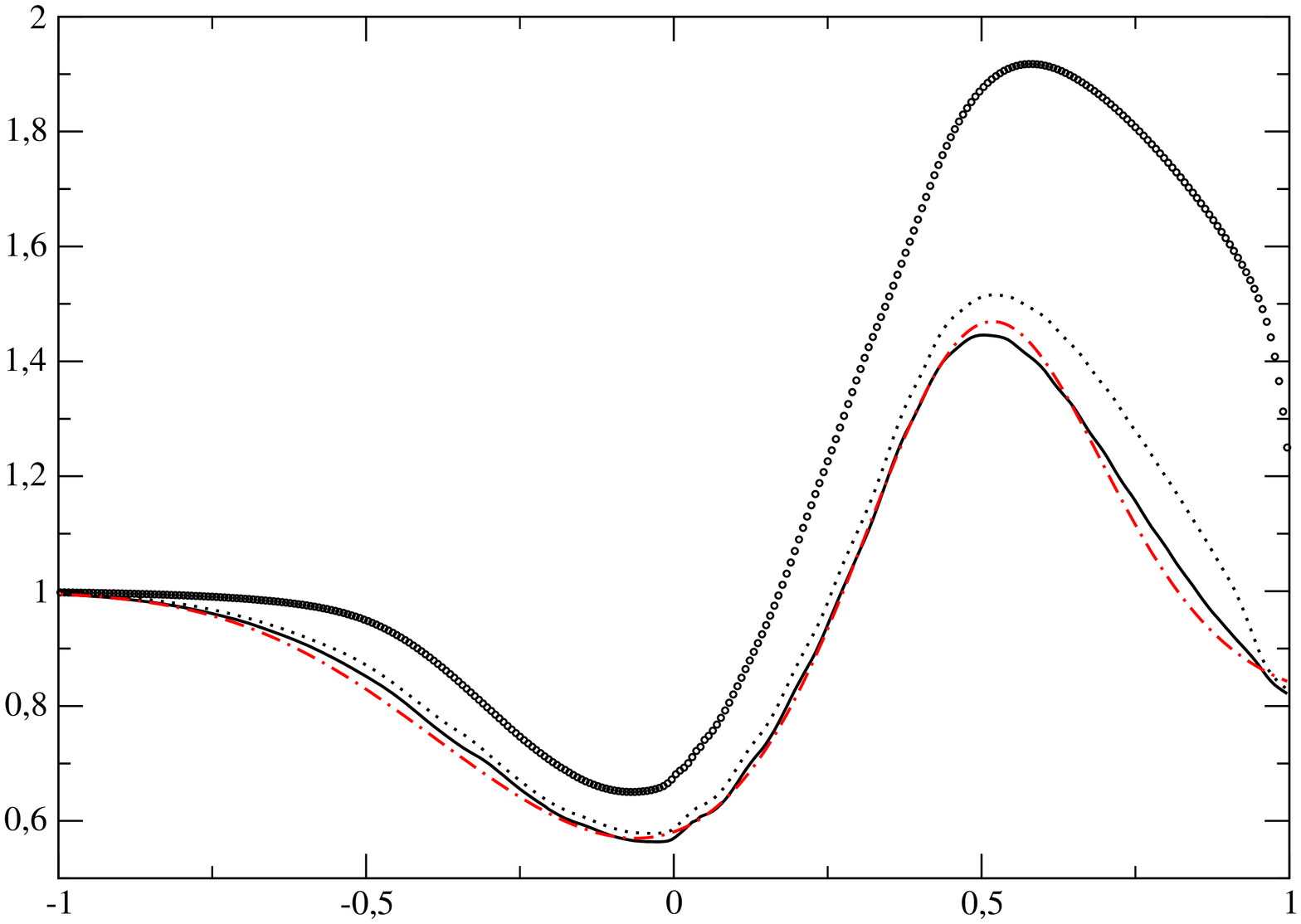}
\caption{Sod test case, free transport: temperature at
  time $0.3$, comparison between
  the exact solution (dot-dashed) and several LDVs with 30 points and
  the {\it moment correction method}:
  with first order interpolation ('o'), with 3 points-ENO (dotted), with 4
  points-ENO (solid). } \label{fig:sod_step4_temperature}
\end{center}
\end{figure}

\clearpage
\begin{figure}[!ht]
\begin{center}
\includegraphics[angle=-90, width=\textwidth]{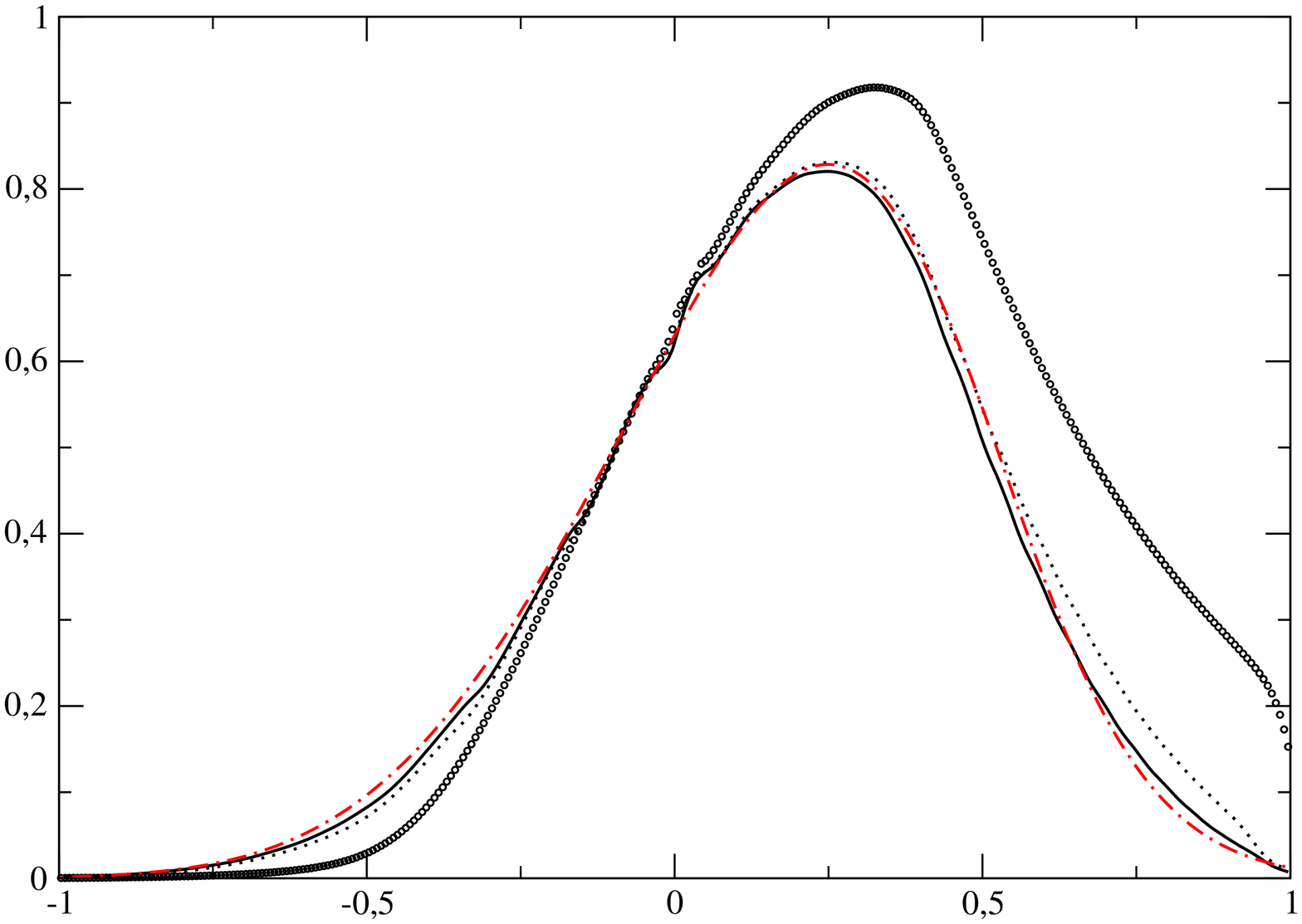}
\caption{Sod test case, free transport: velocity at
  time $0.3$, comparison between
  the exact solution (dot-dashed) and several LDVs with 30 points and
  the  {\it moment correction method}:
  with first order interpolation ('o'), with 3 points-ENO (dotted), with 4
  points-ENO (solid). } \label{fig:sod_step4_velocity}
\end{center}
\end{figure}

\clearpage
\begin{figure}[!ht]
\begin{center}
\includegraphics[angle=-90, width=\textwidth]{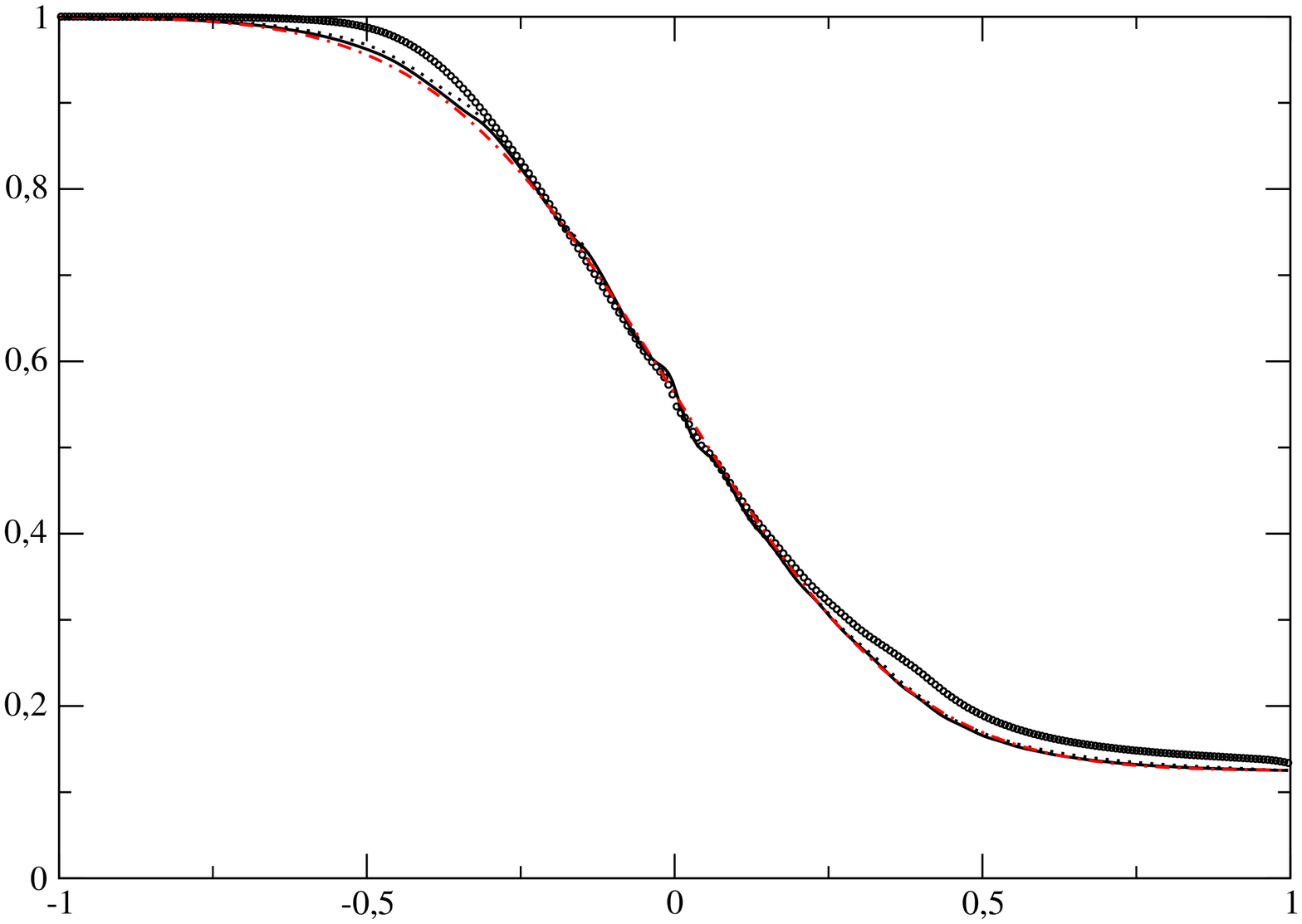}
\caption{Sod test case, free transport: density at
  time $0.3$, comparison between
  the exact solution (dot-dashed) and several LDVs with 30 points and
  the  {\it moment correction method}:
  with first order interpolation ('o'), with 3 points-ENO (dotted), with 4
  points-ENO (solid). } \label{fig:sod_step4_density}
\end{center}
\end{figure}

\clearpage
\begin{figure}[!ht]
\begin{center}
\begin{minipage}[c]{0.75\textwidth}
\subfigure{\includegraphics[angle=-90, width=\textwidth]{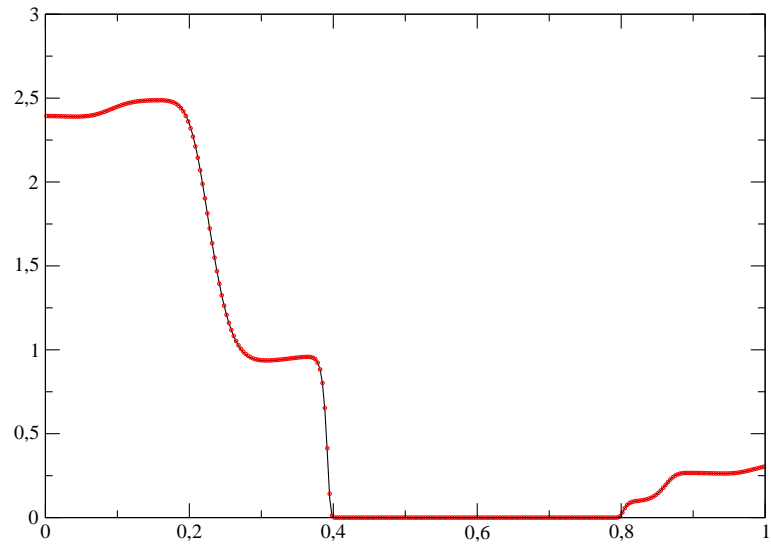}}
\end{minipage}
\begin{minipage}[c]{0.75\textwidth}
\subfigure{\includegraphics[angle=-90,
  width=\textwidth]{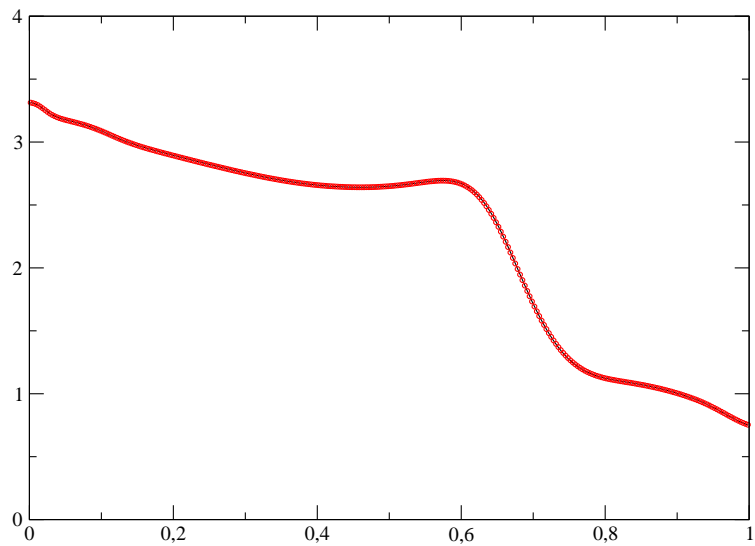}}
\end{minipage}
\caption{``Two interacting blast waves'': temperature, before the shock at time $0.008$ (top) and after the shock at time $0.05 $
  (bottom). The solid line is the solution obtained with the LDV
  method ($30$ points), the dotted line is the solution  
  computed with the DVM method ($2\, 551$ velocities).} \label{temperatureblast}
\end{center}
\end{figure}

\clearpage
\begin{figure}[!ht]
\begin{center}
\begin{minipage}[c]{0.75\textwidth}
\subfigure{
\includegraphics[angle=-90, width=\textwidth]{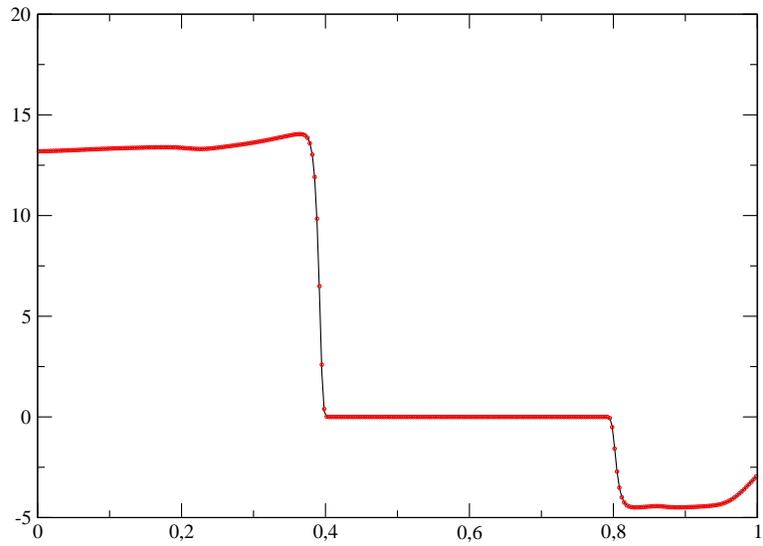}} 
\end{minipage} \hfill
\begin{minipage}[c]{0.75\textwidth}
\subfigure{
\includegraphics[angle=-90, width=\textwidth]{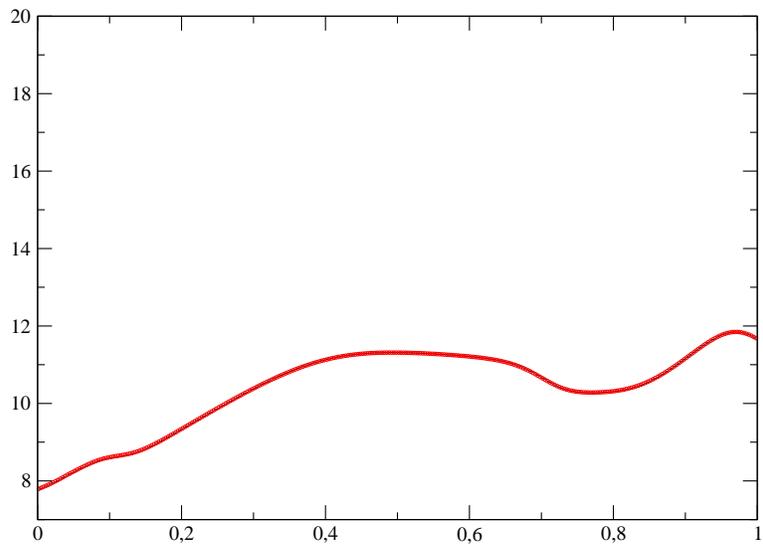}} 
\end{minipage}
\caption{``Two interacting blast waves'': velocity (same as figure~\ref{temperatureblast}).} \label{vitesseblast}
\end{center}
\end{figure}

\clearpage
\begin{figure}[!ht]
\begin{center}
\begin{minipage}[c]{0.75\textwidth}
\subfigure{
\includegraphics[angle=-90, width=\textwidth]{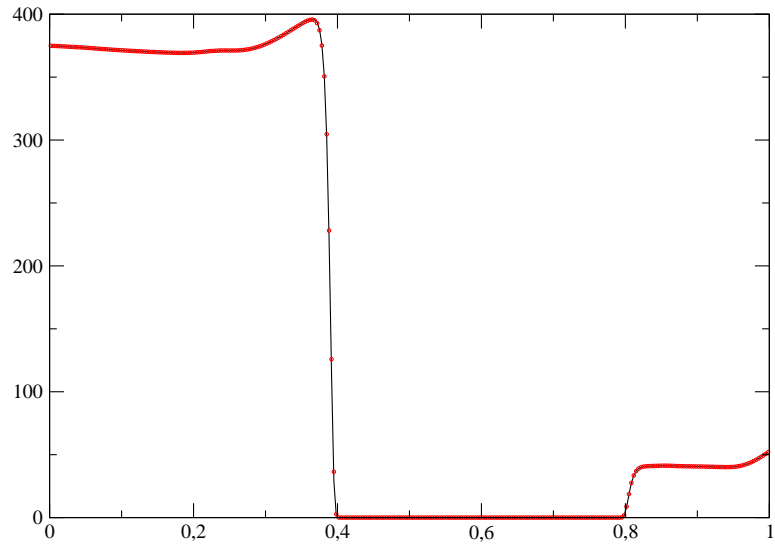}}
\end{minipage}\hfill%
\begin{minipage}[c]{0.75\textwidth}
\subfigure{
\includegraphics[angle=-90, width=\textwidth]{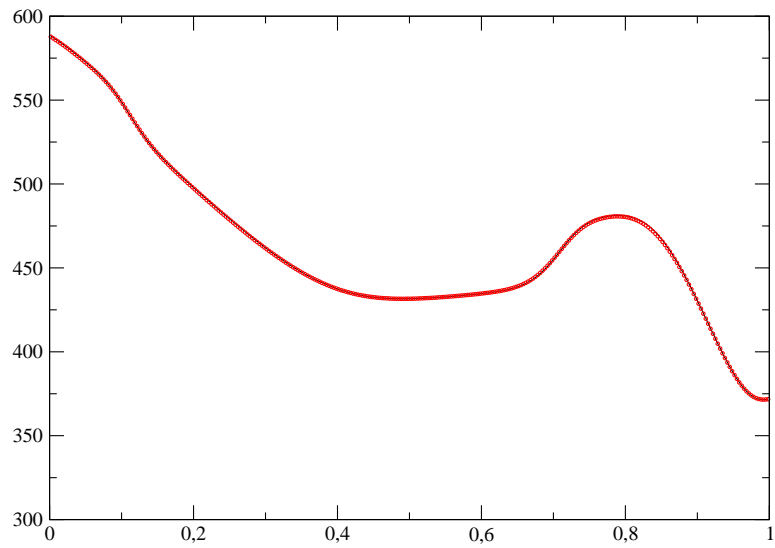}}
\end{minipage}
\caption{``Two interacting blast waves'': pressure (same as figure~\ref{temperatureblast}).} \label{pressionblast}
\end{center}
\end{figure}

\clearpage
\begin{figure}[!ht]
  \centering
  \includegraphics[angle=-90, width=0.75\textwidth]{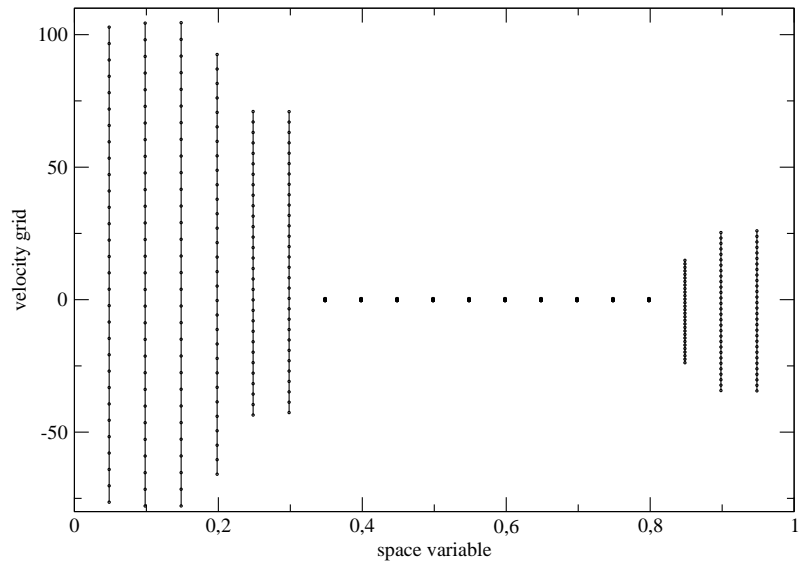}
  \includegraphics[angle=-90, width=0.75\textwidth]{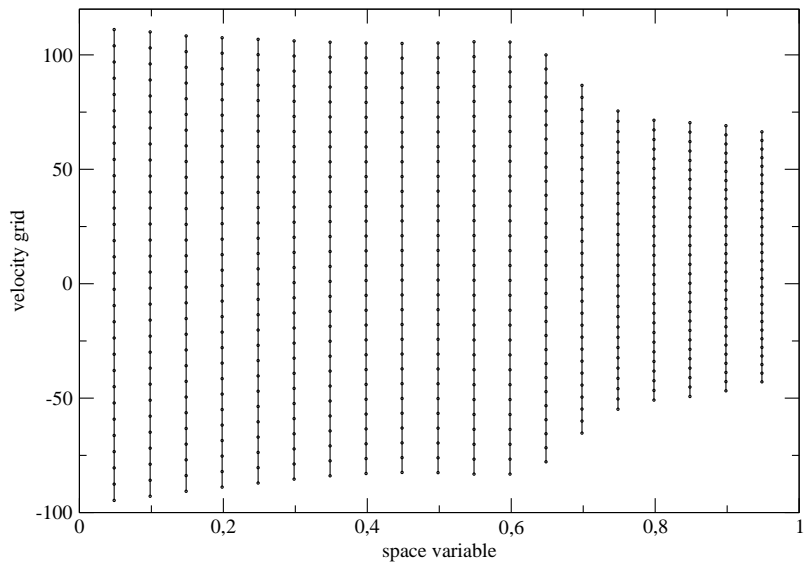}
  \caption{``Two interacting blast waves'': some local velocity grids for different space positions, before the shock at time 0.008 (top) and after the shock at time 0.05 (bottom).}
  \label{fig:blast_grids}
\end{figure}

\clearpage
\begin{figure}[!ht]
\begin{center}
\begin{minipage}[c]{0.7\textwidth}
\subfigure{
\includegraphics[angle=-90,width=\textwidth]{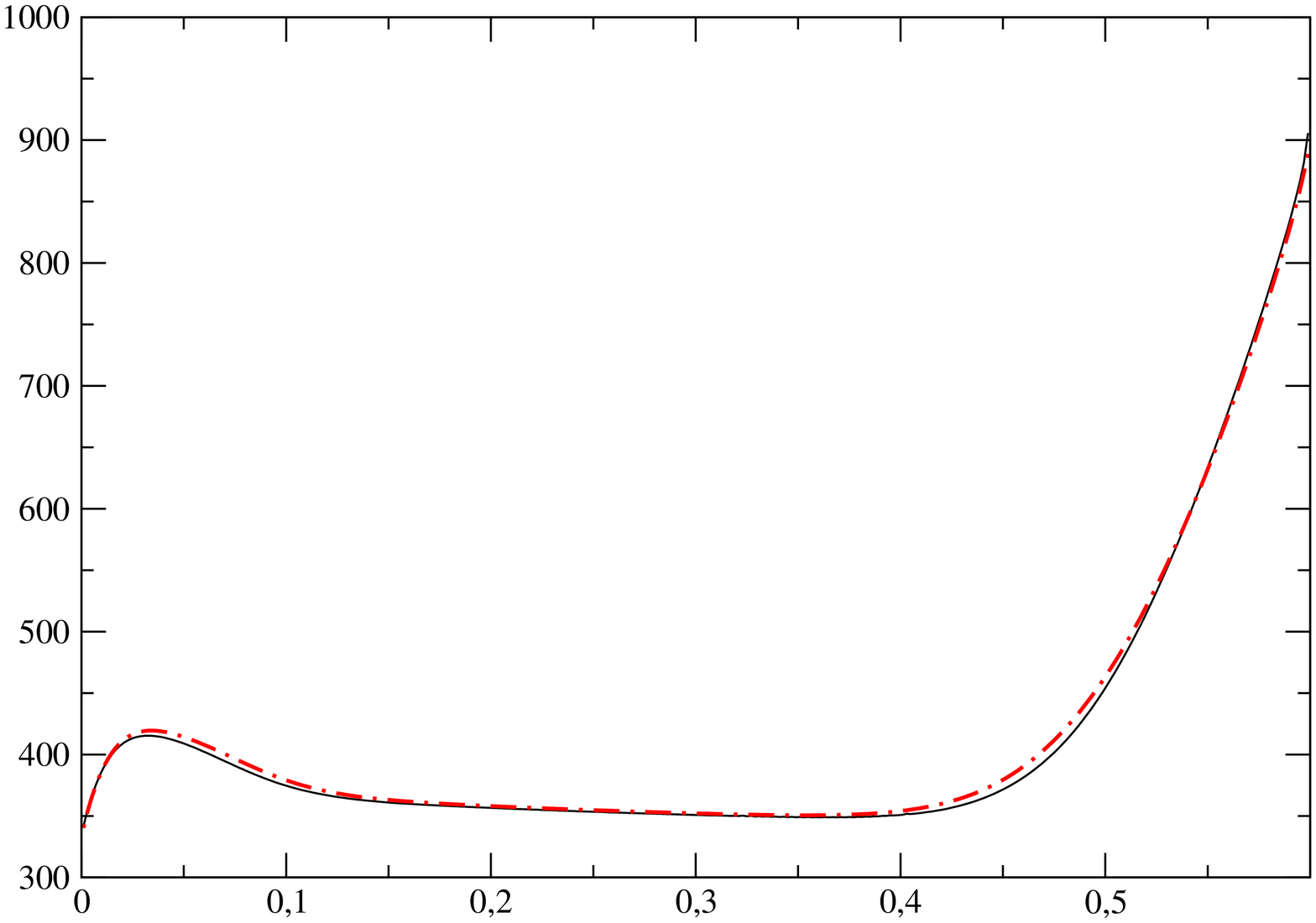}}
\end{minipage}
\begin{minipage}[c]{0.7\textwidth}
\subfigure{
\includegraphics[angle=-90,width=\textwidth]{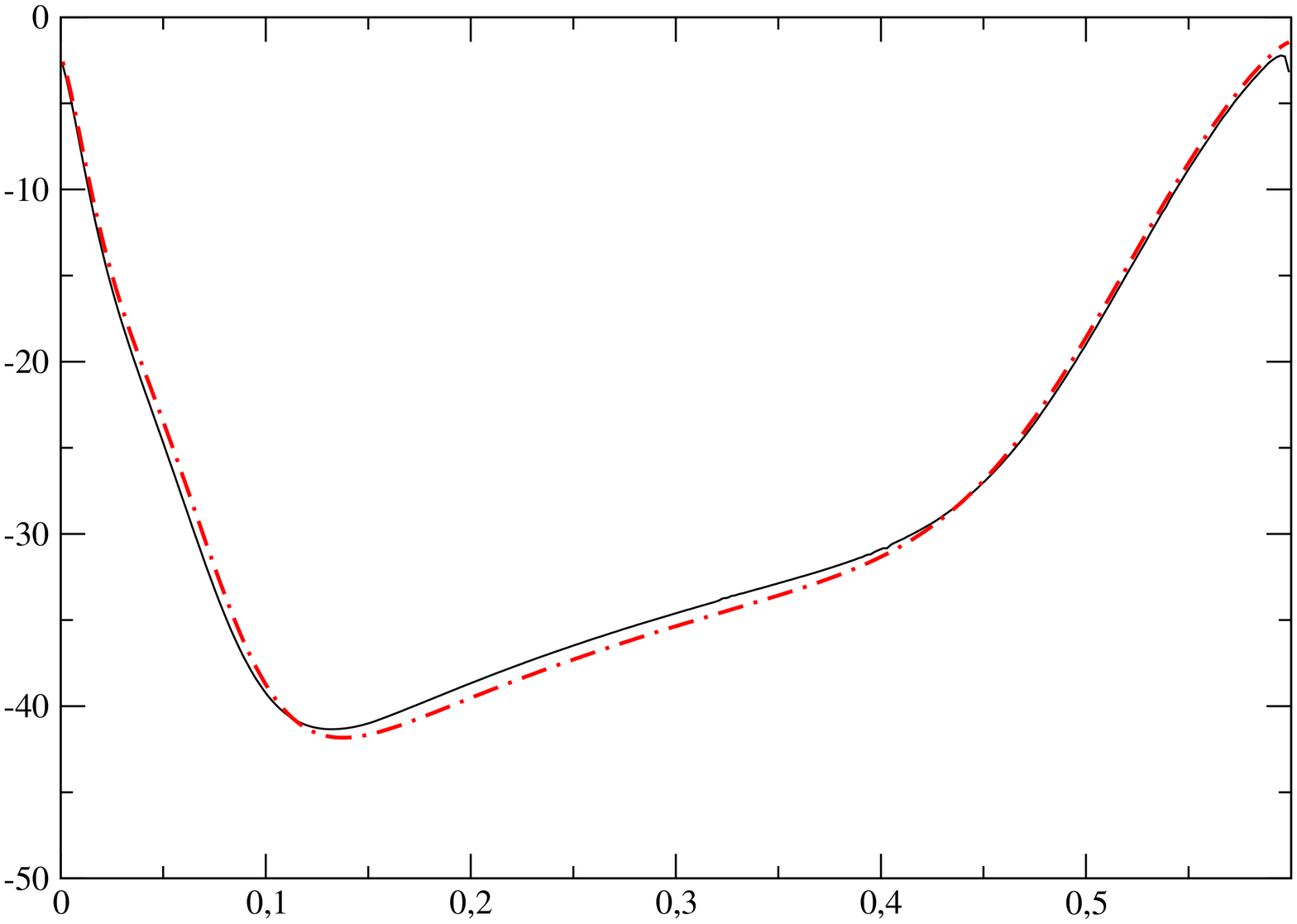}}
\end{minipage}
\caption{Heat transfer problem, transitional regime: temperature (top) and velocity
(bottom) at time $1.3 \, 10^{-3}$, $\Kn = 10^{-2}$. The solid line is the solution given
  by the LDV method ($30$ velocities), the dot-dashed line is the DVM method
  (30 velocities).\label{fig:heat_trans} } 
\end{center}
\end{figure}

\clearpage
\begin{figure}[!ht]
\begin{center}
\begin{minipage}[c]{0.75\textwidth}
\subfigure{
\includegraphics[angle=-90, width=\textwidth]{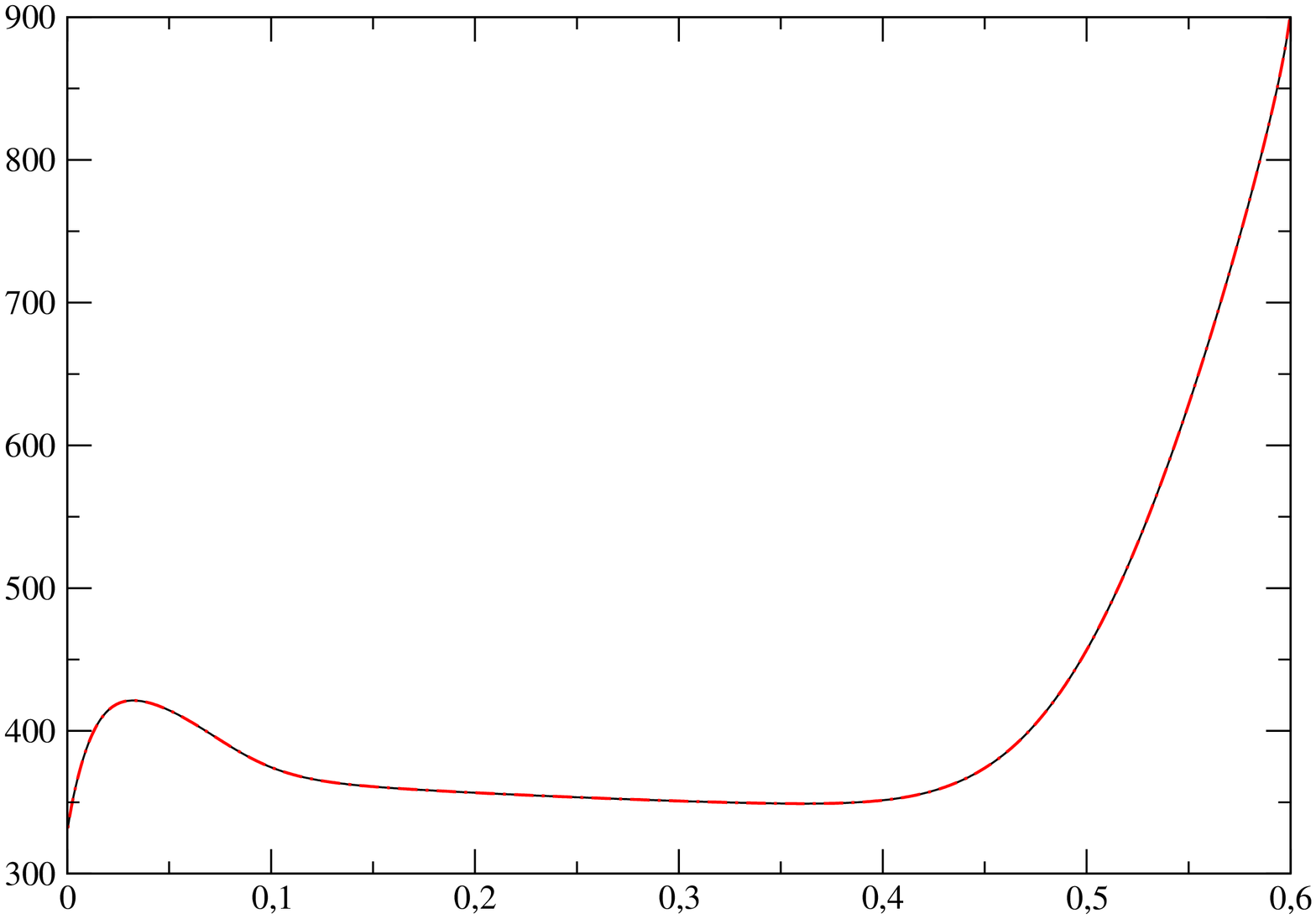}}
\end{minipage} 
\begin{minipage}[c]{0.75\textwidth}
\subfigure{
\includegraphics[angle=-90, width=\textwidth]{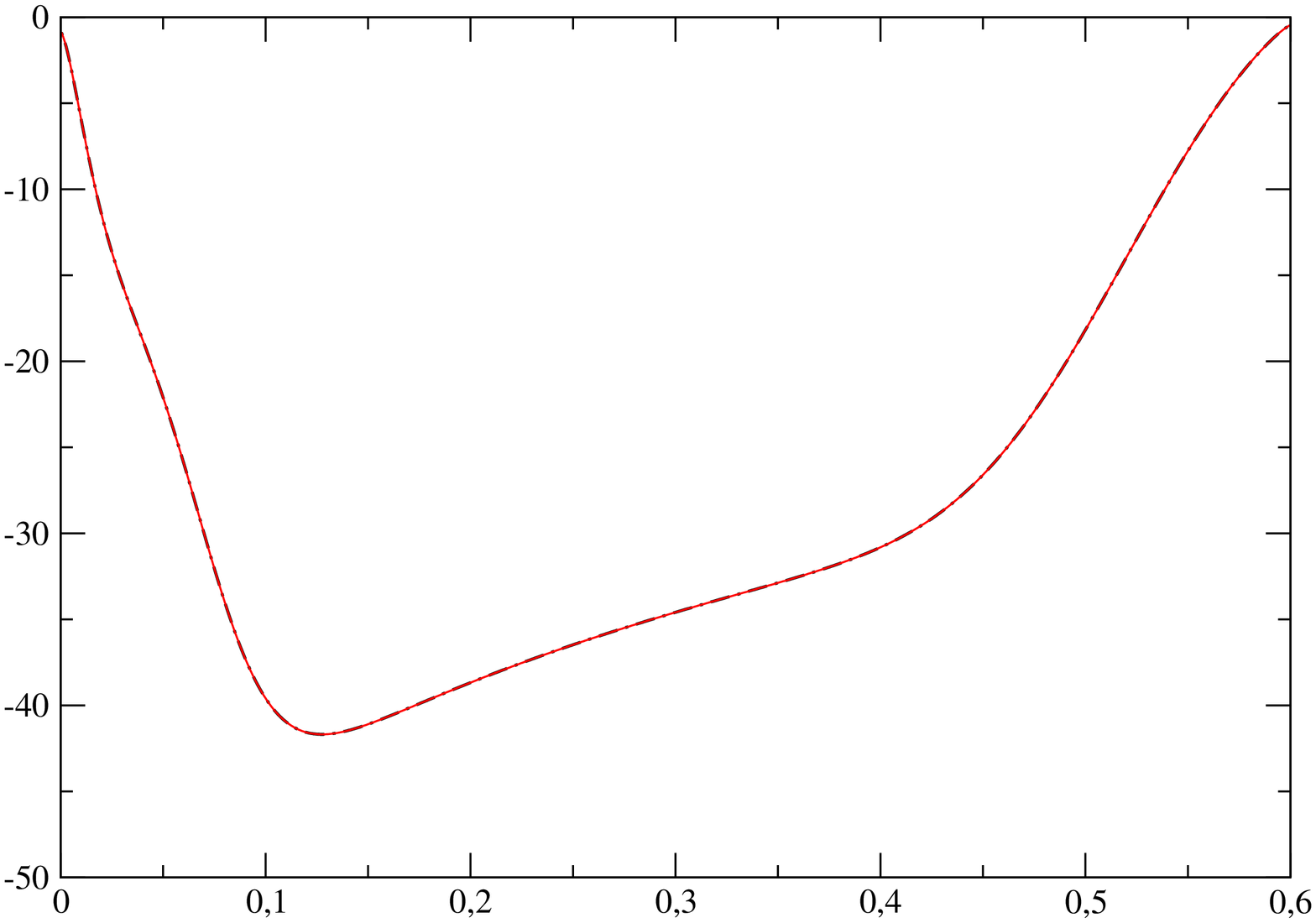}}  
\end{minipage}
\caption{Heat transfer problem, transitional regime (test of the non
  symmetric local grid): temperature (top) and velocity (bottom) at
  time $1.3 \, 10^{-3}$, $\Kn = 10^{-2}$. The solid line is the solution
  given by the LDV method ($30$ velocities), the dot-dashed line is the
  DVM method (30 velocities).  } \label{transfertelargievitesse_etape4}
\end{center}
\end{figure}

\clearpage
\begin{figure}[!ht]
\begin{center}
\begin{minipage}[c]{0.75\textwidth}
\subfigure{
\includegraphics[angle=-90, width=\textwidth]{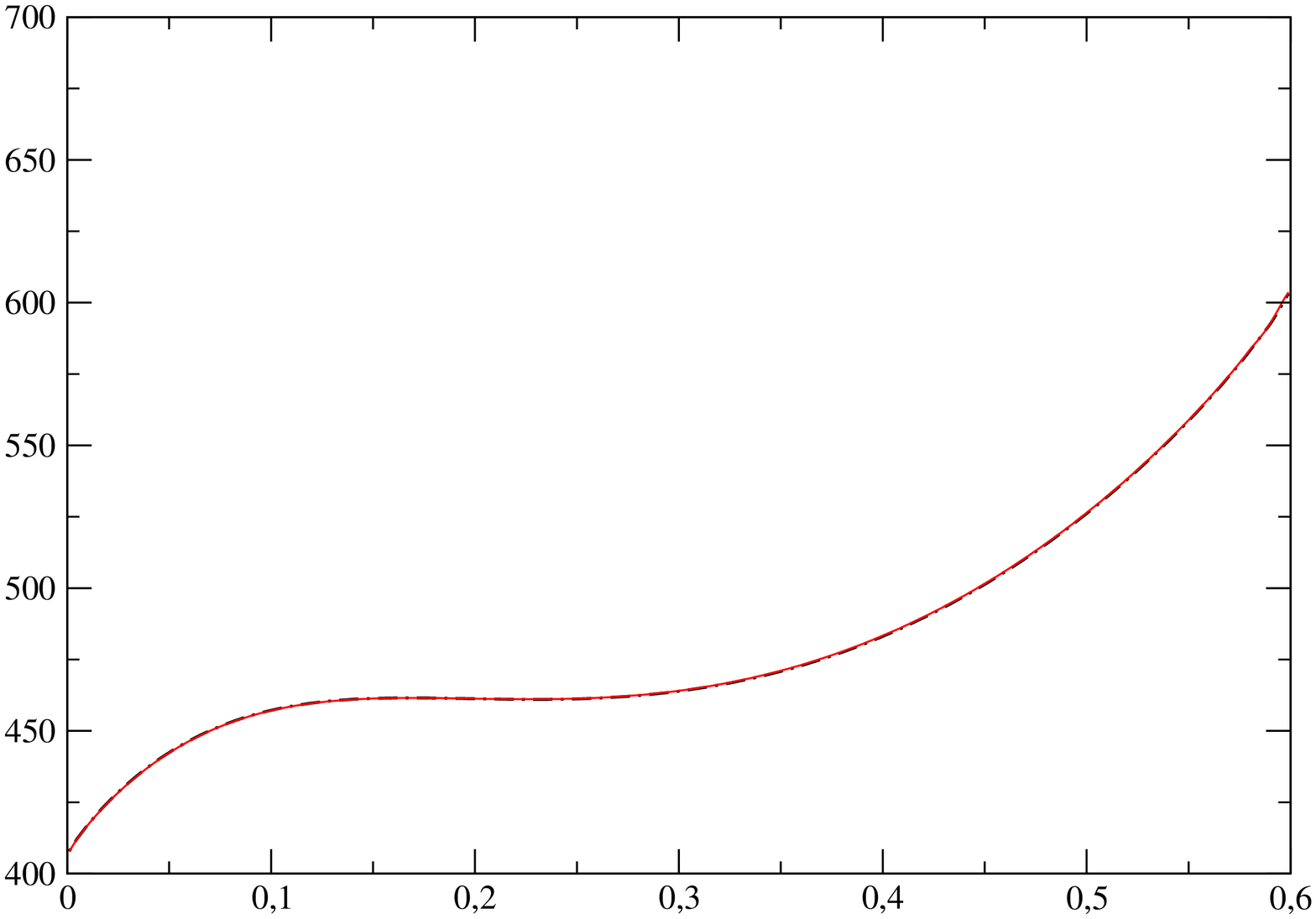}}
\end{minipage} 
\begin{minipage}[c]{0.75\textwidth}
\subfigure{
\includegraphics[angle=-90, width=\textwidth]{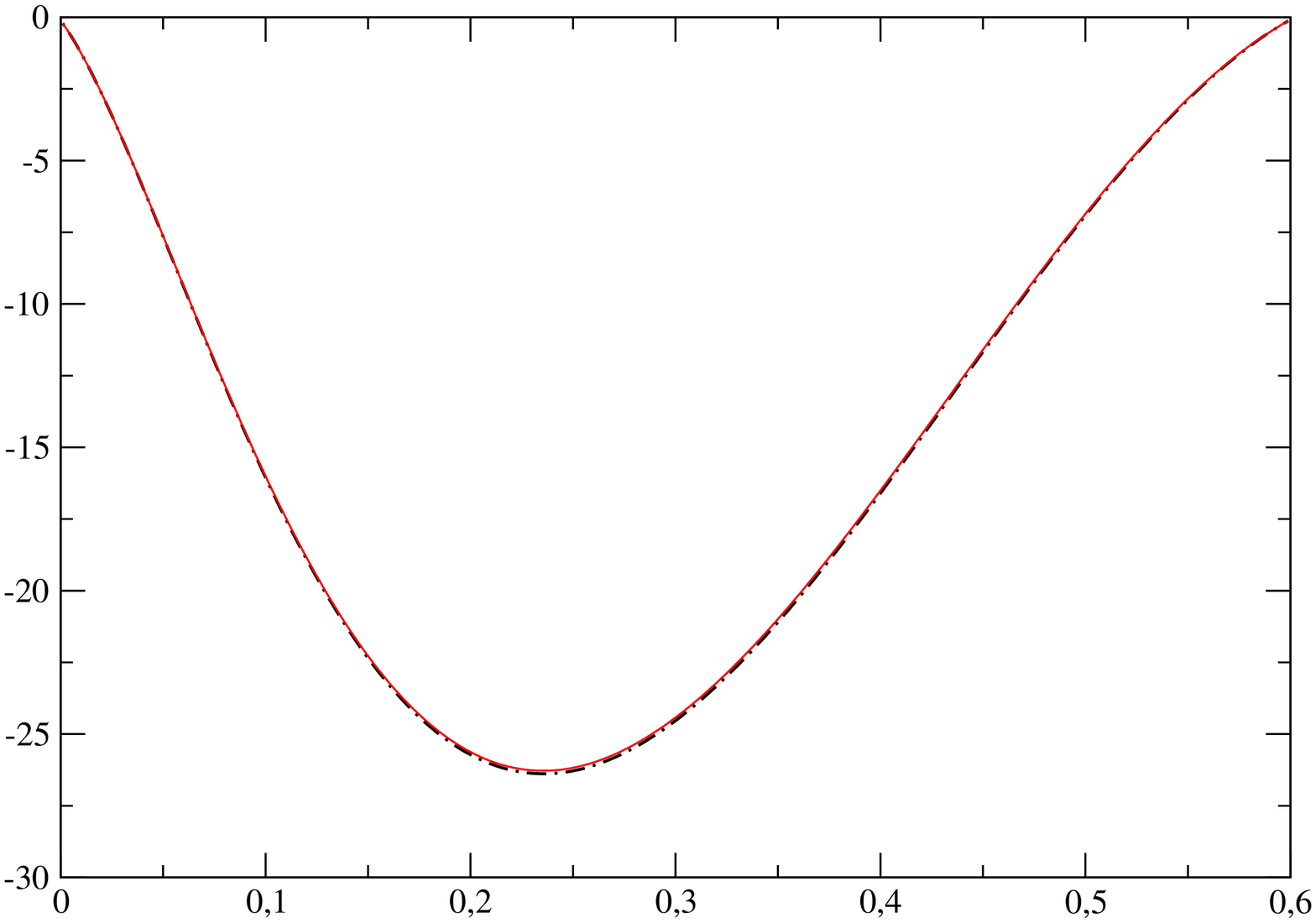}}  
\end{minipage}
\caption{Heat transfer problem, rarefied regime: temperature (top) and velocity (bottom) at
  time $1.3 \, 10^{-3}$, $\Kn = 1$. The solid line is the solution
  given by the LDV method ($300$ velocities, non symmetric local grids), the dot-dashed line is the
  DVM method (300 velocities).  } \label{fig:heat_kn0}
\end{center}
\end{figure}

\clearpage
\begin{figure}[!ht]
\begin{center}
\begin{minipage}[c]{0.75\textwidth}
\subfigure{
\includegraphics[angle=-90, width=\textwidth]{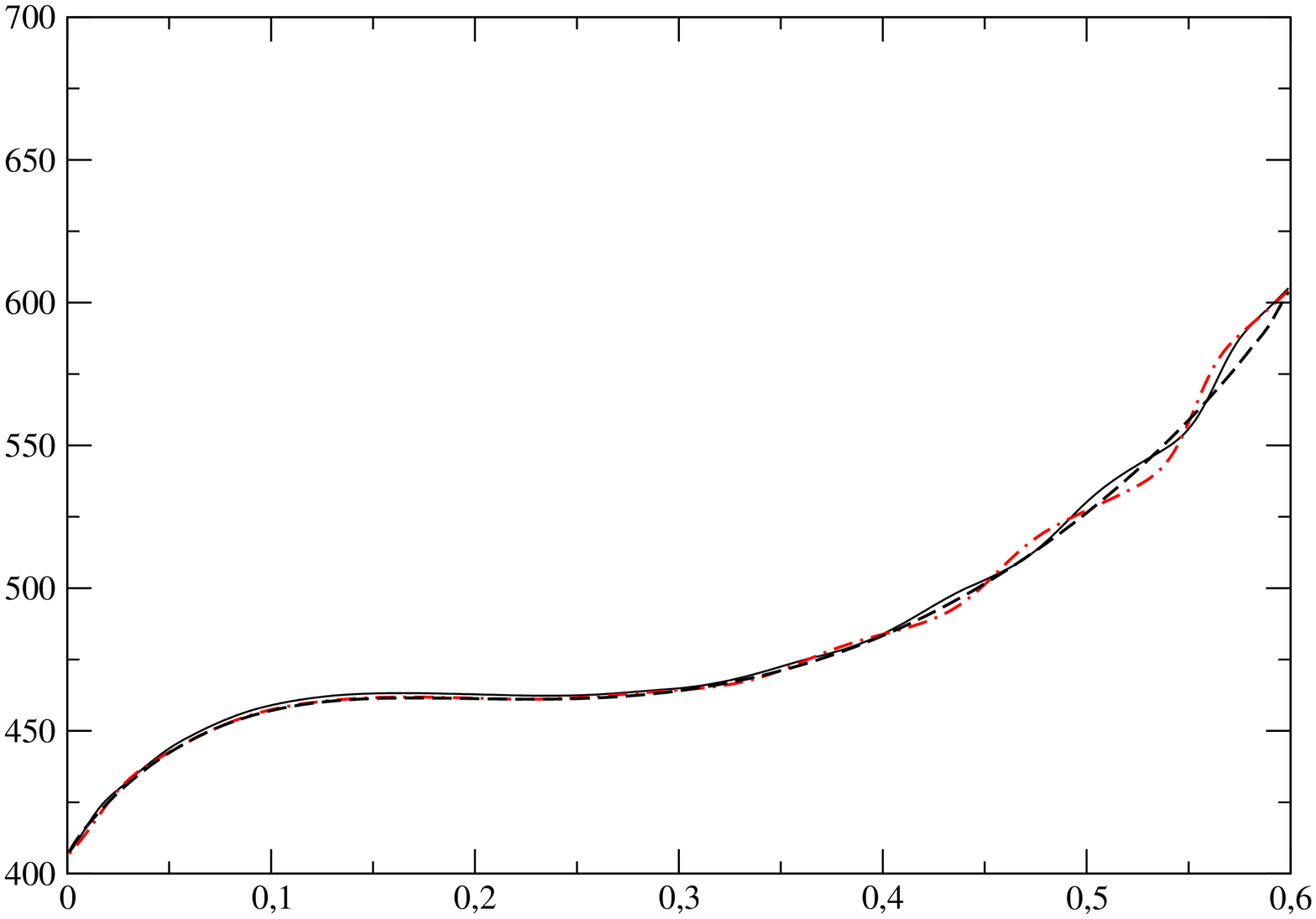}}
\end{minipage} 
\begin{minipage}[c]{0.75\textwidth}
\subfigure{
\includegraphics[angle=-90, width=\textwidth]{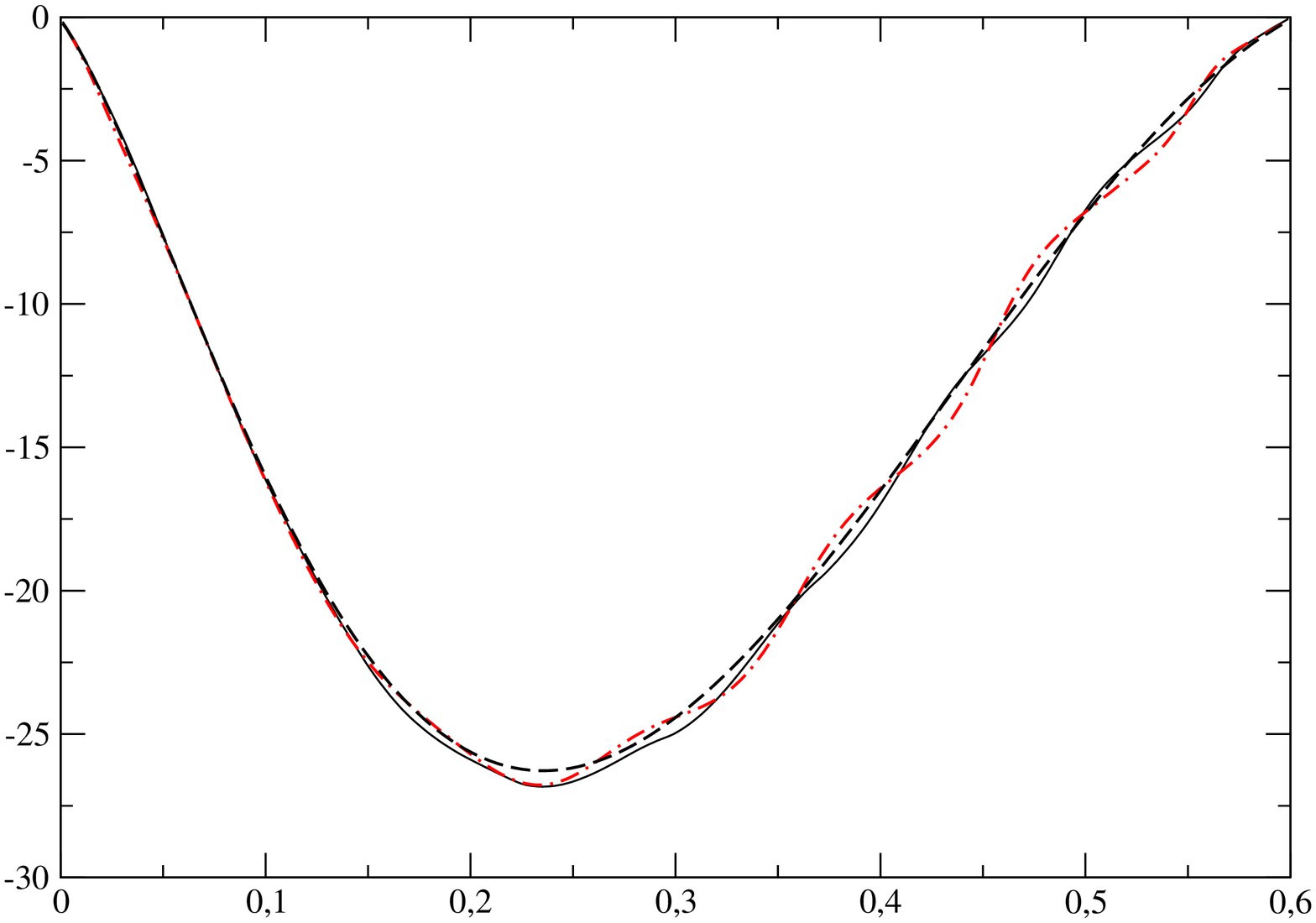}}  
\end{minipage}
\caption{Heat transfer problem, rarefied regime: temperature (top) and velocity (bottom) at
  time $1.3 \, 10^{-3}$, $\Kn = 1$. The solid line is the solution
  given by the LDV method ($50$ velocities, non symmetric local grids), the dot-dashed line is the
  DVM method (50 velocities), the dashed line is the DVM method with
  300 velocities (reference solution).  } \label{fig:heat_kn0_50}
\end{center}
\end{figure}

\clearpage
\begin{figure}[!ht]
\begin{center}
\begin{minipage}[c]{0.75\textwidth}
\subfigure{
\includegraphics[angle=-90, width=\textwidth]{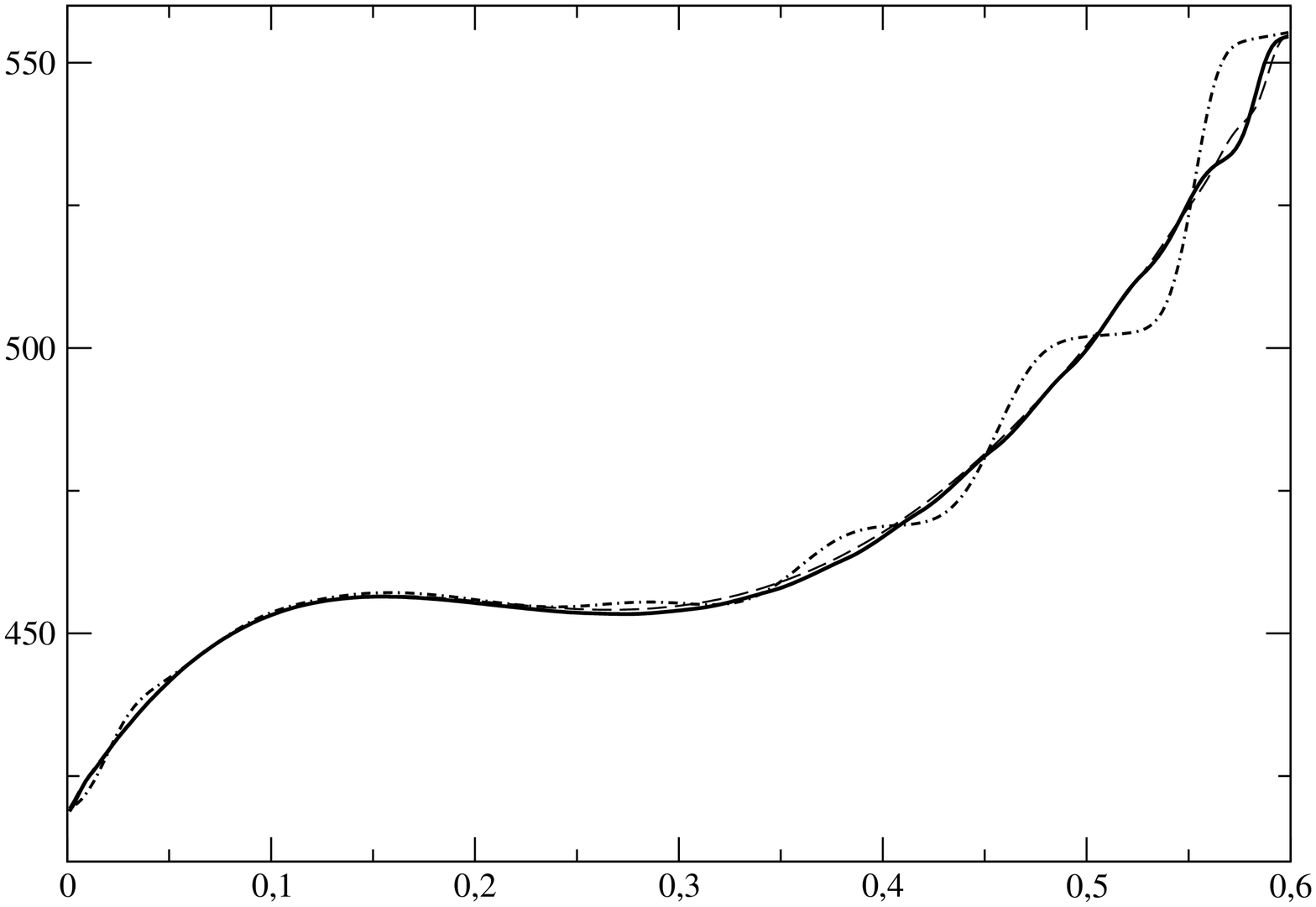}}
\end{minipage} 
\begin{minipage}[c]{0.75\textwidth}
\subfigure{
\includegraphics[angle=-90, width=\textwidth]{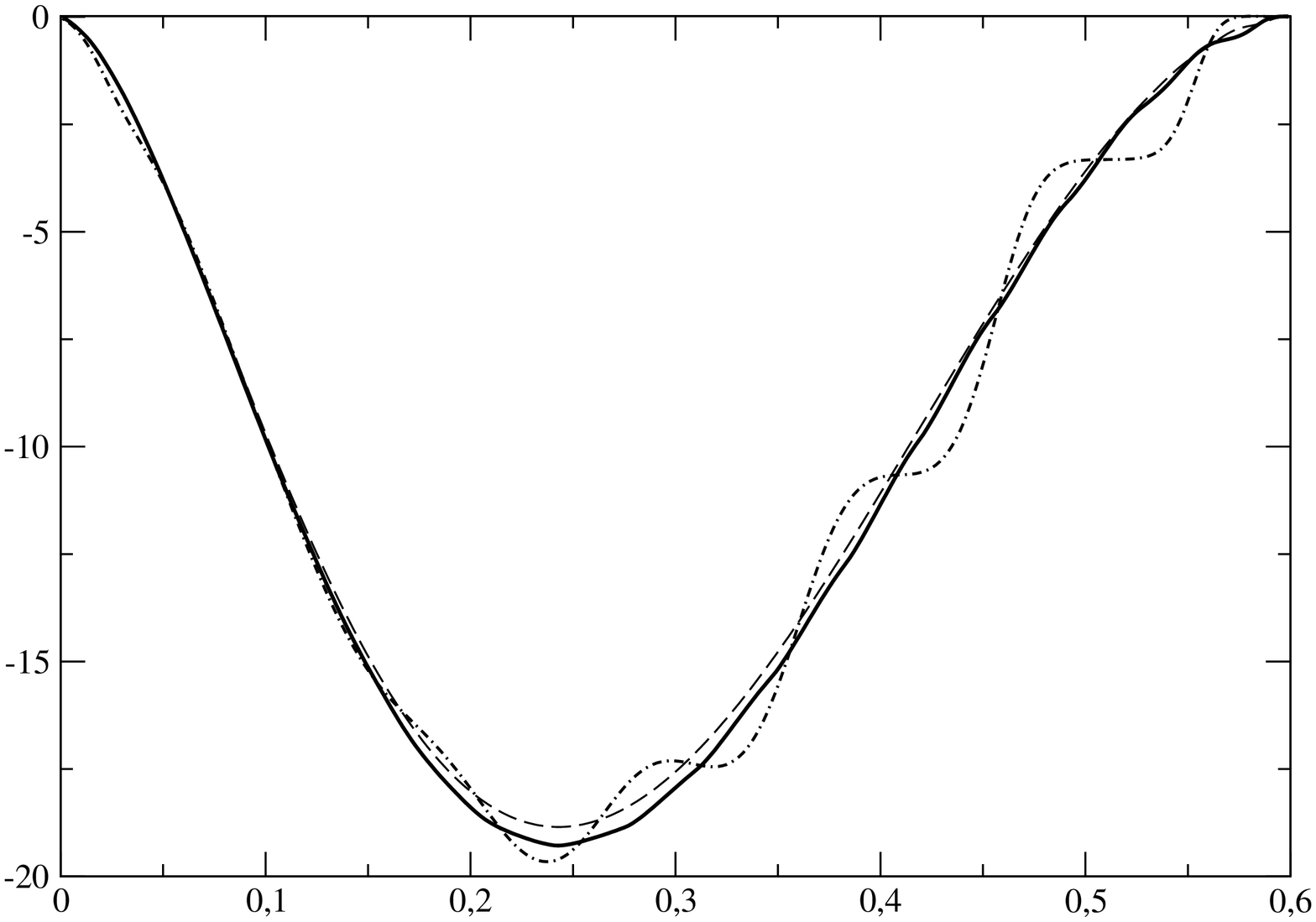}}  
\end{minipage}
\caption{Heat transfer problem, rarefied regime: temperature (top) and velocity (bottom) at
  time $1.3 \, 10^{-3}$, $\Kn = 10$. The solid line is the LDV method with
  100 velocities, the dot-dashed line is the DVM method (100 velocities), the dashed line is the solution
  given by the DVM method ($400$ velocities, reference solution).  } \label{fig:heat_kn10}
\end{center}
\end{figure}

\clearpage 

\begin{figure}[!ht]
\begin{center}
\begin{minipage}[c]{0.75\textwidth}
\subfigure{
\includegraphics[angle=-90, width=\textwidth]{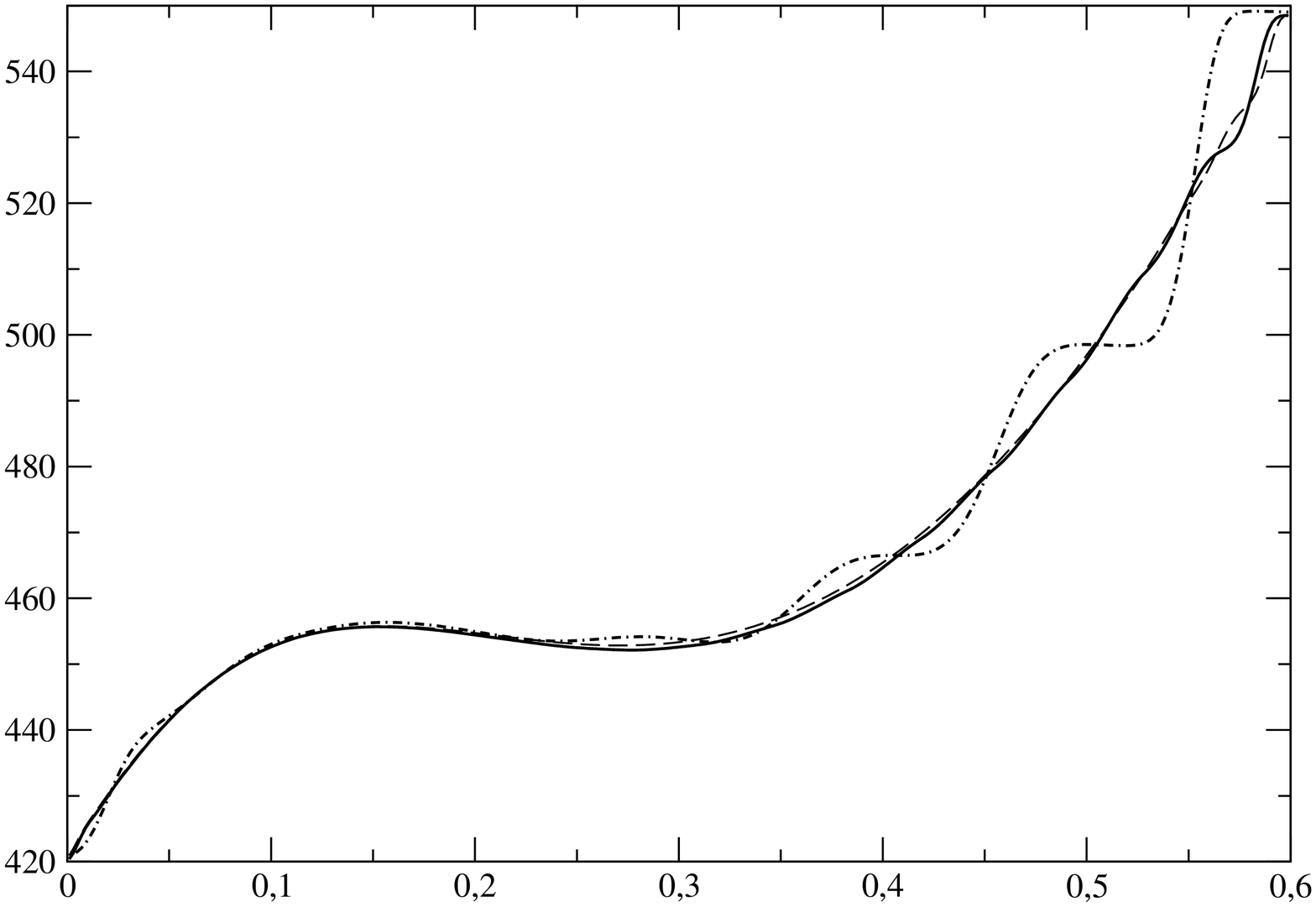}}
\end{minipage} 
\begin{minipage}[c]{0.75\textwidth}
\subfigure{
\includegraphics[angle=-90, width=\textwidth]{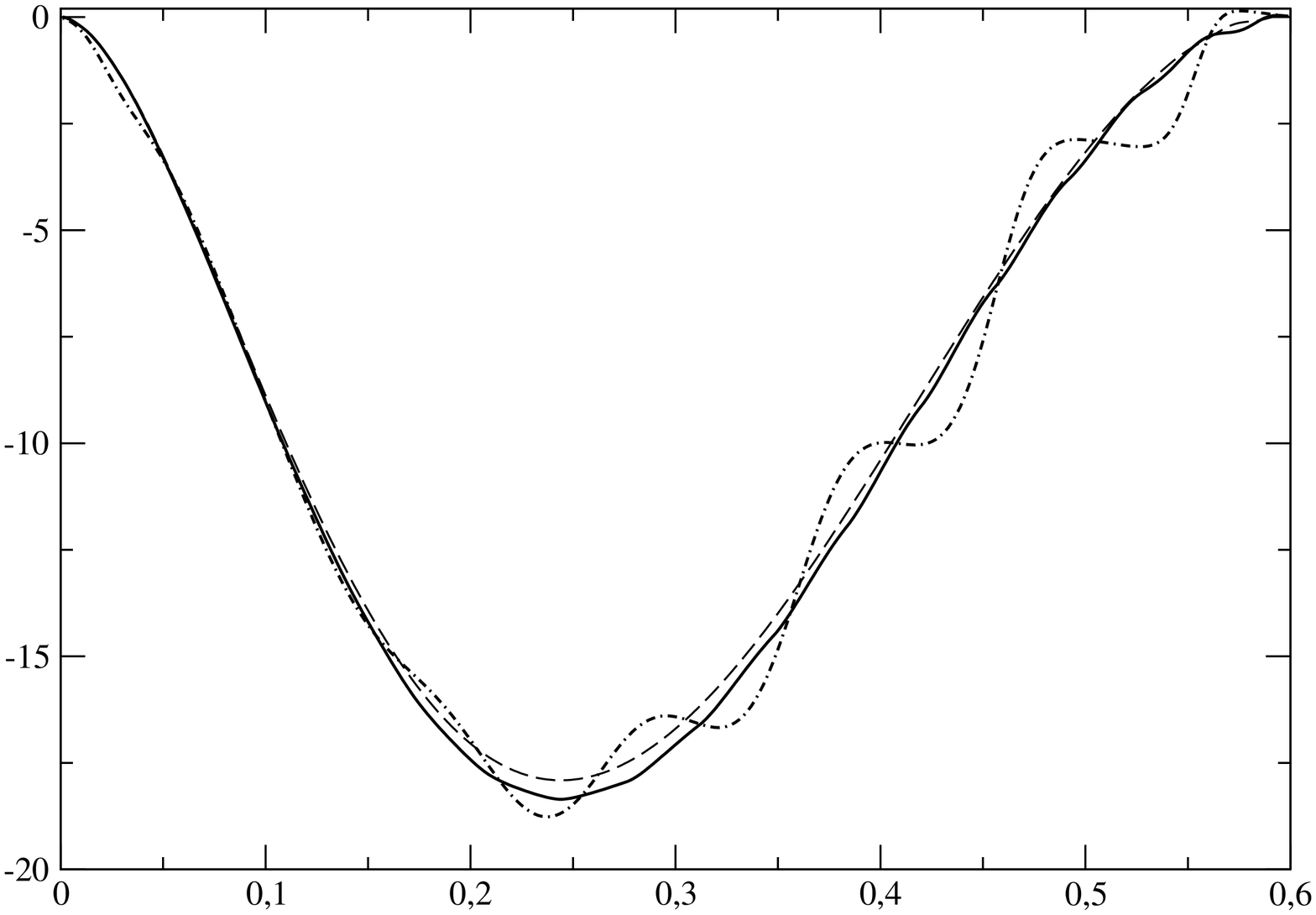}}  
\end{minipage}
\caption{Heat transfer problem, rarefied regime: temperature (top) and velocity (bottom) at
  time $1.3 \, 10^{-3}$, $\Kn = 1000$. The solid line is the LDV method with
  100 velocities, the dot-dashed line is the DVM method (100 velocities), the dashed line is the solution
  given by the DVM method ($400$ velocities, reference solution).  } \label{fig:heat_kn1000}
\end{center}
\end{figure}

\clearpage
\begin{table}[!ht]
  \centering
  \begin{tabular}[c]{|l|c|c|c|c|c|}
    \hline
     & Sod   
& Blast waves
& Heat transfer  \\
\hline
DVM & 0.344 & 949 & 64 \\
\hline
LDV  & 5.136 &  20 & 14 \\
\hline
  \end{tabular}
  \caption{CPU time comparisons (in seconds) between the LDV and DVM methods}
  \label{table:cpu}
\end{table}

\end{document}